\documentclass[11pt,twoside,reqno]{amsart}

\usepackage{microtype}
\usepackage[OT1]{fontenc}
\usepackage{type1cm}
\usepackage{amssymb}
\usepackage{mathtools}
\usepackage{esint}
\usepackage[left=2.3cm,top=3cm,right=2.3cm]{geometry}

\usepackage[active]{srcltx}
\usepackage{verbatim}

\usepackage[pdftex]{graphicx}
\usepackage{subfigure}

\numberwithin{equation}{section}

\theoremstyle{plain}
\newtheorem{theorem}{Theorem}[section]
\newtheorem{corollary}[theorem]{Corollary}
\newtheorem{proposition}[theorem]{Proposition}
\newtheorem{lemma}[theorem]{Lemma}

\theoremstyle{remark}
\newtheorem{remark}[theorem]{Remark}

\newtheorem{example}[theorem]{Example}
\newtheorem*{ack}{Acknowledgement}

\theoremstyle{definition}

\newcommand{\HH}{\mathcal{H}}
\newcommand{\PP}{\mathcal{P}}

\newcommand{\R}{\mathbb{R}}

\newcommand{\Q}{\mathbb{Q}}

\newcommand{\N}{\mathbb{N}}
\newcommand{\hhh}{\mathtt{h}}
\newcommand{\iii}{\mathtt{i}}
\newcommand{\jjj}{\mathtt{j}}
\newcommand{\kkk}{\mathtt{k}}
\newcommand{\eps}{\varepsilon}
\newcommand{\fii}{\varphi}
\newcommand{\roo}{\varrho}
\newcommand{\ualpha}{\overline{\alpha}}
\newcommand{\lalpha}{\underline{\alpha}}
\newcommand{\ur}{\overline{r}}
\newcommand{\lr}{\underline{r}}
\newcommand{\la}{\langle}
\newcommand{\ra}{\rangle}

\newcommand{\dd}{\,\mathrm{d}}

\newenvironment{labeledlist}[2][\unskip]
{ \renewcommand{\theenumi}{({#2}\arabic{enumi}{#1})}
  \renewcommand{\labelenumi}{({#2}\arabic{enumi}{#1})}
  \begin{enumerate} }
{ \end{enumerate} }

\DeclareMathOperator{\ldimloc}{\underline{dim}_{loc}}

\DeclareMathOperator{\dimh}{dim_H}

\DeclareMathOperator{\ldimh}{\underline{dim}_H}

\DeclareMathOperator{\dima}{dim_A}

\DeclareMathOperator{\dist}{dist}
\DeclareMathOperator{\diam}{diam}

\DeclareMathOperator{\spt}{spt}

\begin{document}

\title{Weak separation condition, Assouad dimension, and Furstenberg homogeneity}

\author{Antti K\"aenm\"aki}
\author{Eino Rossi}
\address{University of Jyvaskyla \\
         Department of Mathematics and Statistics \\
         P.O. Box 35 (MaD) \\
         FI-40014 University of Jyv\"askyl\"a \\
         Finland}
\email{antti.kaenmaki@jyu.fi}
\email{eino.rossi@jyu.fi}

\thanks{ER acknowledges the support of the Vilho, Yrj{\"o}, and Kalle V{\"a}is{\"a}l{\"a} foundation}
\date{\today}

\begin{abstract}
  We consider dimensional properties of limit sets of Moran constructions satisfying the finite clustering property. Just to name a few, such limit sets include self-conformal sets satisfying the weak separation condition and certain sub-self-affine sets. In addition to dimension results for the limit set, we manage to express the Assouad dimension of any closed subset of a self-conformal set by means of the Hausdorff dimension. As an interesting consequence of this, we show that a Furstenberg homogeneous self-similar set in the real line satisfies the weak separation condition. We also exhibit a self-similar set which satisfies the open set condition but fails to be Furstenberg homogeneous.
\end{abstract}

\maketitle

\section{Introduction}

Moran constructions and dimensional properties of their limit sets have been studied extensively in Euclidean spaces; for example, see \cite{FengWenWu1997, HollandZhang2013, HuaRaoWenWu2000, KaenmakiVilppolainen2008, LiWu2011, Moran1946}. Particular examples of such sets are self-similar, self-conformal, and self-affine sets, and their modifications. Our main goal in this article is to help the progress of the dimension theory on these sets. We have chosen to present the results in a setting as general as possible. This serves two purposes. In general metric spaces there often are no non-trivial self-similar sets but Moran constructions occur naturally. Therefore it is justifiable to seek for results in metric spaces; see \cite{BaloghRohner2007, BaloghTysonWarhurst2009, KaenmakiRajalaSuomala2012b, RajalaVilppolainen2013}. Secondly, working in a general setting often helps to uncover simpler proofs. This becomes apparent in Remark \ref{rem:wsc}(2).

In the center of our considerations is the finite clustering property; see \S \ref{sec:fcp} for the definition. Roughly speaking, it is for Moran constructions what the open set condition is for self-similar and self-conformal sets. Perhaps a bit surprisingly, the finite clustering property is also related to the weak separation condition introduced by Lau and Ngai \cite{LauNgai1999}. In Proposition \ref{thm:wsc}, we show that an iterated function system satisfying the weak separation condition introduces a Moran construction satisfying the finite clustering property.

Assuming the finite clustering property, we show in Proposition \ref{thm:pressure_dim} that the Assouad dimension and the Hausdorff dimension of the limit set of a Moran construction coincide. Furthermore, the dimension can be obtained as a zero of an appropriate pressure function. When studying limit sets of Moran constructions, measures are often pushed from the corresponding shift space. Therefore, it is not guaranteed that all the characteristics of the measure are preserved. In Proposition \ref{thm:dimloc}, we show that, assuming the finite clustering property, the local dimension of the push-forward measure can be determined directly in the shift space.

An interesting question is to try to say something about the dimension of a subset of the limit set. Of course, for a generic Moran construction this is most likely an impossible task. Nevertheless, in Proposition \ref{thm:dimf}, we show that the Assouad dimension of a closed subset of a self-conformal set is the supremum of Assouad dimensions of its microsets. A microset is a limit of symbolic magnifications of the self-conformal set; see \S \ref{sec:microset} for the definition. On a self-homothetic set, they correspond to the limits of Euclidean magnifications; see Remark \ref{rem:microset_is_microset}. In Corollary \ref{thm:cifs-micro-dimensiot}, assuming the finite clustering property, we improve this result so that the supremum can be taken from Hausdorff dimensions. The progression from Proposition \ref{thm:dimf} to Corollary \ref{thm:cifs-micro-dimensiot} is highly non-trivial and it is based on dynamical tools called CP-processes. CP-processes were introduced by Furstenberg \cite{Furstenberg2008}, greatly developed by Hochman \cite{Hochman2010}, and studied by K\"aenm\"aki, Sahlsten, and Shmerkin \cite{KaenmakiSahlstenShmerkin2015}. They have been proven to be a useful tool in the study of geometric properties of sets and measures; see, for example, \cite{FraserPollicott2015, KaenmakiSahlstenShmerkin2015b, Kempton2015}. We review the required theory on CP-processes in \S \ref{sec:distributions}.

We finish the article by studying Furstenberg homogeneous sets. Roughly speaking, a set is Furstenberg homogeneous if the limits of its Euclidean magnifications are always contained in some Euclidean magnification of the set. For a precise definition, see \S \ref{sec:micro_and_dim}. We apply Corollary \ref{thm:cifs-micro-dimensiot} and a recent result of Fraser, Henderson, Olson, and Robinson \cite{FraserHendersonOlsonRobinson2015} to show that a Furstenberg homogeneous self-similar set in the real line satisfies the weak separation condition. Although a self-homothetic set satisfying the strong separation condition is Furstenberg homogeneous, the weak separation condition does not characterize Furstenberg homogeneity. We show this by exhibiting a self-homothetic set which satisfies the open set condition but fails to be Furstenberg homogeneous.

\section{Setting and preliminaries}

In this section, we give the basic definitions for our study. In particular, we recall how the shift space can be used to model various sets arising from Moran constructions and iterated function systems.

\subsection{Shift space}
Let $\kappa \ge 2$ and $\Sigma = \{ 1,\ldots,\kappa \}^\N$ be the collection of all infinite words constructed from integers $\{ 1,\ldots,\kappa \}$. We denote the left shift operator by $\sigma$ and equip $\Sigma$ with the usual ultrametric in which the distance between two different words is $2^{-n}$, where $n$ is the first place at which the words differ. The \emph{shift space} $\Sigma$ is clearly compact. If $\iii = i_1i_2\cdots \in \Sigma$, then we define $\iii|_n = i_1 \cdots i_n$ for all $n \in \N$. The empty word $\iii|_0$ is denoted by $\varnothing$. For $\Gamma \subset \Sigma$ we set $\Gamma_n = \{ \iii|_n : \iii \in \Gamma \}$ for all $n \in \N$ and $\Gamma_* = \bigcup_{n=0}^\infty \Gamma_n$. Thus $\Sigma_*$ is the free monoid on $\Sigma_1 = \{ 1,\ldots,\kappa \}$. The concatenation of two words $\iii \in \Sigma_*$ and $\jjj \in \Sigma_* \cup \Sigma$ is denoted by $\iii\jjj$.

The length of $\iii \in \Sigma_* \cup \Sigma$ is denoted by $|\iii|$. For $\iii
\in \Sigma_*$ we set $\iii^- = \iii|_{|\iii|-1}$ and $[\iii] = \{ \iii\jjj \in
\Sigma : \jjj \in \Sigma \}$. The set $[\iii]$ is called a \emph{cylinder set}.
Cylinder sets are open and closed and they generate the Borel $\sigma$-algebra.
If $\iii,\jjj \in \Sigma_*$ such that $[\iii] \cap [\jjj] = \emptyset$, then we
write $\iii \bot \jjj$. The longest common prefix of $\iii,\jjj \in \Sigma_*
\cup \Sigma$ is denoted by $\iii \wedge \jjj$. Thus $\iii = (\iii \wedge
\jjj)\iii'$ and $\jjj = (\iii \wedge \jjj)\jjj'$ for some $\iii',\jjj' \in
\Sigma_* \cup \Sigma$.

For each set $R \subset \Sigma_*$ of forbidden words we define
\begin{equation*}
  \Sigma[R] = \Sigma \setminus \{ \iii\kkk\jjj : \iii \in \Sigma_*, \; \kkk \in R, \text{ and } \jjj \in \Sigma \} = \bigcap_{n=1}^\infty \bigcup_{\hhh \in \Sigma[R]_n} [\hhh].
\end{equation*}
The set $\Sigma[R]$ is compact and satisfies $\sigma(\Sigma[R]) \subset \Sigma[R]$ for all $R \subset \Sigma_*$. Conversely, if $\Gamma \subset \Sigma$ is a compact set satisfying $\sigma(\Gamma) \subset \Gamma$, then it is straightforward to see that there exists $R \subset \Sigma_*$ such that $\Gamma = \Sigma[R]$.

\subsection{Moran construction}
Let $X$ be a complete metric space, $\Sigma$ a shift space, and $\Gamma \subset \Sigma$ a compact set satisfying $\sigma(\Gamma) \subset \Gamma$. We assume that there is a collection $\{ E_\iii : \iii \in \Gamma_* \}$ of closed and bounded subsets of $X$ with positive diameter. Such a collection is called a \emph{Moran construction} if \ref{M1} and \ref{M2} of the following four conditions are satisfied:
\begin{labeledlist}{M}
  \item $E_\iii \subset E_{\iii^-}$ for all $\iii \in \Gamma_* \setminus \{ \varnothing \}$, \label{M1}
  \item $\diam(E_{\iii|_n}) \to 0$ as $n \to \infty$ for all $\iii \in \Gamma$, \label{M2}
  \item there exists $D \ge 1$ such that $\diam(E_{\iii\jjj}) \le D\diam(E_\iii)\diam(E_\jjj)$ for all $\iii\jjj\in \Gamma_*$, \label{M3}
  \item there exists $0<\lalpha<1$ such that $\diam(E_\iii) \ge \lalpha\diam(E_{\iii^-})$ for all $\iii \in \Gamma_*
\setminus \{ \varnothing \}$, \label{M4}
\end{labeledlist}
Observe that if $\iii,\jjj \in \Sigma_*$ are such that the concatenation $\iii\jjj \in \Gamma_*$, then $\sigma(\Gamma) \subset \Gamma$ implies that both $\iii$ and $\jjj$ are in $\Gamma_*$.

Given a Moran construction, the \emph{projection mapping} $\pi \colon \Gamma \to X$ is defined by setting
\begin{equation*}
  \{ \pi(\iii) \} = \bigcap_{n=1}^\infty E_{\iii|_n}
\end{equation*}
for all $\iii \in \Gamma$. The assumptions \ref{M1}, \ref{M2}, and the fact that the underlying metric space is complete guarantee that this is a well-defined continuous mapping. The compact set $\pi(\Gamma)$ is called the \emph{limit set} of the Moran construction and throughout the paper, we shall denote it by $E$.

A finite collection $\{ \fii_i \}_{i=1}^\kappa$ of injective Lipschitz contractions $\fii_i \colon X \to X$ is called an \emph{iterated function system} if there are two mappings $\fii_i$ and $\fii_j$ having distinct fixed points. A mapping $\fii_i$ is a \emph{Lipschitz contraction} if there is $0<\ur_i<1$ so that
\begin{equation} \label{lip_ehto}
  d(\fii_i(x),\fii_i(y)) \le \ur_i d(x,y)
\end{equation}
for all $x,y \in X$. Furthermore, it is a \emph{bi-Lipschitz contraction} if there are $0<\lr_i\le \ur_i<1$ so that
\begin{equation} \label{bilip_ehto}
  \lr_i d(x,y) \le d(\fii_i(x),\fii_i(y)) \le \ur_i d(x,y)
\end{equation}
for all $x,y \in X$. Here $d$ is the metric of $X$. If all the mappings $\fii_i$ are bi-Lipschitz contractions, then the iterated function system $\{ \fii_i \}_{i=1}^\kappa$ is called \emph{bi-Lipschitz}. We write $\fii_\iii = \fii_{i_1} \circ \cdots \circ \fii_{i_n}$ for all $\iii = i_1 \cdots i_n \in \Sigma_n$ and $n \in \N$.

Iterated function systems naturally give rise to Moran constructions. The next lemma makes this precise. The proof follows the lines of the proof of the existence of the invariant set of the iterated function system; see \cite{Hutchinson1981}. Of course, once the invariant set is shown to exist, it can be used in place of $W$ to produce a Moran construction with the same limit set.

\begin{lemma} \label{thm:ifs}
  If $\{ \fii_i \}_{i=1}^\kappa$ is an iterated function system on a complete metric space $X$, then there exists a set $W \subset X$ such that for each compact set $\Gamma \subset \Sigma$ satisfying $\sigma(\Gamma) \subset \Gamma$ the collection $\{ \fii_\iii(W) : \iii \in \Gamma_* \}$ is a Moran construction. Furthermore, if the iterated function system $\{ \fii_i \}_{i=1}^\kappa$ is bi-Lipschitz and there is a constant $C \ge 1$ so that
  \begin{equation} \label{BDP_ehto}
    d(\fii_\iii(x),\fii_\iii(y)) \le C\diam(\fii_\iii(W)) d(x,y)
  \end{equation}
  for all $x,y \in W$ and $\iii \in \Gamma_*$, then the Moran construction
  satisfies \ref{M3} and \ref{M4}.
\end{lemma}

\begin{proof}
  As in \cite[3.1(5)]{Hutchinson1981}, let $z_i$ be the fixed point of $\fii_i$ and define $W = \bigcap_{i \in \{ 1,\ldots,\kappa \}} B(z_i,R)$, where $R = \lambda(1-\ualpha)^{-1}$, $\ualpha = \max_{i \in \{ 1,\ldots,\kappa \}} \ur_i$, and $\lambda = \max_{i,j \in \{ 1,\ldots,\kappa \}} d(z_i,z_j)$. Here $B(x,r)$ denotes the closed ball centered at $x \in X$ with radius $r>0$. Oberve that $W \subset X$ is closed and bounded with positive diameter. Thus $\fii_\iii(W)$ is a closed and bounded set with positive diameter for all $\iii \in \Gamma_*$. It is now straightforward to see that the required properties are satisfied.
\end{proof}

\begin{example} \label{ex:moran}
  If $\{ \fii_i \}_{i=1}^\kappa$ is an iterated function system on a complete metric space $X$, then the limit set $E=\pi(\Sigma)$ of the Moran construction of Lemma \ref{thm:ifs} is the unique non-empty compact invariant set satisfying
  \begin{equation*}
    E = \bigcup_{i=1}^\kappa \fii_i(E).
  \end{equation*}
  If the mappings $\fii_i$ are homothetic, similitudes, conformal, or affine (defined on an appropriate space), then $E$ is called \emph{self-homothetic}, \emph{self-similar}, \emph{self-conformal}, or \emph{self-affine}, respectively. Furthermore, if $\Gamma \subset \Sigma$ is a compact set satisfying $\sigma(\Gamma) \subset \Gamma$, then $E=\pi(\Gamma)$ satisfies
  \begin{equation*}
    E \subset \bigcup_{i=1}^\kappa \fii_i(E)
  \end{equation*}
  and is called \emph{sub-self-homothetic}, \emph{sub-self-similar}, \emph{sub-self-conformal}, or \emph{sub-self-affine}, respectively; see \cite{Falconer1995, KaenmakiVilppolainen2010}. It is straightforward to see that all of these particular choices of mappings satisfy \eqref{BDP_ehto} and thus, the corresponding Moran construction of Lemma \ref{thm:ifs} satisfies \ref{M3} and \ref{M4}.
\end{example}

The following lemma shows that, by relying on \ref{M1}, \ref{M3}, and the compactness of $\Gamma$, the assumption \ref{M2} can be improved. Its proof follows easily from \cite[Lemma 3.1]{KaenmakiVilppolainen2008} and the proof of \cite[Proposition 4.10]{KaenmakiVilppolainen2008}. This form of \ref{M2} is used frequently in our considerations.

\begin{lemma} \label{upperalpha}
  Suppose that $X$ is a complete metric space and $\Gamma \subset \Sigma$ is a compact set satisfying $\sigma(\Gamma) \subset \Gamma$. If $\{ E_\iii : \iii \in \Gamma_* \}$ is a Moran construction satisfying \ref{M3}, then there exist constants $C \ge 1$ and $0<\ualpha<1$ so that
  \begin{equation*}
    \max_{\iii \in \Gamma_n} \diam(E_\iii) \le C\ualpha^n
  \end{equation*}
  for all $n \in \N$.
\end{lemma}

\section{Separation conditions and dimension}

In this section, we continue the study of the finite clustering property. For a self-conformal set, the finite clustering property is equivalent to the open set condition; see e.g.\ \cite[Corollary 5.8]{KaenmakiVilppolainen2008}. An iterated function system $\{ \fii_i \}_{i=1}^\kappa$ satisfies the \emph{open set condition} if there exists a non-empty open set $U$ such that $\fii_i(U) \subset U$ for all $i$ and $\fii_i(U) \cap \fii_j(U) \ne \emptyset$ for $i \ne j$. The open set condition is a classical separation condition introduced already in \cite{Hutchinson1981}. Although the open set condition is often easier to verify, there are classical examples where it is more convenient to consider the finite clustering property directly; see \cite[Example 6.5]{KaenmakiVilppolainen2008}. Also, from a theoretical point of view, the finite clustering property could be considered to be more important since in general complete metric spaces the open set condition is not necessarily a relevant condition anymore; see \cite[Example 4.5]{RajalaVilppolainen2013}.

We will now further emphasize the importance of the finite clustering property. In Proposition \ref{thm:wsc}, we show that the invariant set of an iterated function system satisfying the weak separation is the limit of an appropriate Moran construction satifying the finite clustering property. In Propositions \ref{thm:pressure_dim}, \ref{thm:dimloc}, and \ref{thm:dimf}, we study how various dimensions on Moran constructions can be determined under the finite clustering property. Proposition \ref{thm:dima_cifs} shows how microsets and the Assouad dimension are related in the self-conformal case.

\subsection{Finite clustering property} \label{sec:fcp}
Suppose that $X$ is a complete metric space and $\Gamma \subset \Sigma$ a compact set satisfying $\sigma(\Gamma) \subset \Gamma$. For a given Moran construction $\{ E_\iii : \iii \in \Gamma_* \}$ we define
\begin{equation*}
  \Gamma(r) = \{ \iii \in \Gamma_* : \diam(E_\iii) \le r < \diam(E_{\iii^-}) \}
\end{equation*}
for all $r>0$. Observe that two distinct $\iii, \jjj \in \Gamma(r)$ satisfy $\iii \bot \jjj$ and each $E_\iii$ with $\iii \in \Gamma(r)$ is a set of diameter
roughly $r$.
We also set
\begin{equation*}
  \Gamma(x,r) = \{ \iii \in \Gamma(r) : E_\iii \cap B(x,r) \ne \emptyset \}
\end{equation*}
for all $x \in X$ and $r>0$. Let $E$ be the limit set of the Moran construction. We say that the Moran construction satisfies the \emph{finite clustering property} if
\begin{equation*}
  \sup_{x \in E} \limsup_{r \downarrow 0} \#\Gamma(x,r) < \infty
\end{equation*}
and the \emph{uniform finite clustering property} if
\begin{equation*}
  \sup_{x \in E} \sup_{r>0} \#\Gamma(x,r) < \infty.
\end{equation*}
Since an iterated function system introduces a Moran construction we use the same terminology in relation with iterated function systems and their invariant sets. It is worthwhile to note that, for a self-conformal set, the finite clustering property is equivalent to the uniform finite clustering property; see \cite[Lemma 5.2 and Theorem 3.9]{KaenmakiVilppolainen2008}. It would be interesting to know whether they are equivalent in the self-affine case.

The space $X$ is \emph{doubling} if there exists a constant $N \in \N$ so that every ball of radius $2r$ can be covered by $N$ balls of radius $r$. Considering the restriction metric, the definition extends to all subsets of $X$.

\begin{lemma} \label{thm:ufcp_doubling}
  Suppose that $X$ is a complete metric space, $\Gamma \subset \Sigma$ is a compact set satisfying $\sigma(\Gamma) \subset \Gamma$, and $E$ is the limit set of a Moran construction $\{ E_\iii : \iii \in \Gamma_* \}$ satisfying \ref{M3} and the finite clustering property. Then for each $N \ge 1$ there exists a constant $C \ge 1$ such that for every $x \in E$ there is $r_x>0$ so that
  \begin{equation*}
    \#\{ \iii \in \Gamma(r) : E_\iii \cap B(x,Nr) \ne \emptyset \} \le C
  \end{equation*}
  for all $0<r<r_x$. Furthermore, if the Moran construction satisfies the uniform finite clustering property, then $r_x$ above can be chosen to be infinity for all $x \in E$ and the limit set $E$ is doubling.
\end{lemma}

\begin{proof}
  Let $M \in \N$ be such that for each $x \in E$ there is $r_x>0$ so that $\#\Gamma(x,r) \le M$ for all $0<r<r_x$. Fix $x \in E$ and $r>0$. Recalling Lemma \ref{upperalpha}, choose $n \in \N$ so that $\diam(E_\jjj) \le (ND)^{-1}$ for all $\jjj \in \Gamma_n$. Let $\iii \in \Gamma(x,Nr)$. Since, by \ref{M3}, $\diam(E_{\iii\jjj}) \le D\diam(E_\iii)\diam(E_\jjj) \le r$ whenever $\jjj \in \Gamma_n$ is so that $\iii\jjj \in \Gamma_*$ we have $\#\{ \jjj \in \Gamma_* : \iii\jjj \in \Gamma(r) \} \le (\#\Gamma_1)^n$. Therefore
  \begin{equation*}
    \#\{ \hhh \in \Gamma(r) : E_\hhh \cap B(x,Nr) \ne \emptyset \} \le (\#\Gamma_1)^n M
  \end{equation*}
  and $E \cap B(x,Nr)$ can be covered by $(\#\Gamma_1)^n M$ balls of radius $r$.
\end{proof}

It is evident that in the definition of the finite clustering property, the supremum over $x \in E$ can be replaced by the supremum over $x \in X$. The following lemma shows that this is also the case with the uniform finite clustering property.

\begin{lemma} \label{thm:sup_over_x}
  Suppose that $X$ is a complete metric space, $\Gamma \subset \Sigma$ is a compact set satisfying $\sigma(\Gamma) \subset \Gamma$, and $E$ is the limit set of a Moran construction $\{ E_\iii : \iii \in \Gamma_* \}$ satisfying \ref{M3} and the uniform finite clustering property. Then
  \begin{equation*}
    \sup_{x \in X} \sup_{r>0} \#\Gamma(x,r) < \infty.
  \end{equation*}
\end{lemma}

\begin{proof}
  Fix $x \in X \setminus E$. Since $E$ is compact there is $y \in E$ so that $d(x,y) = \dist(x,E)$. Notice that if $r<\dist(x,E)/2$, then $E_\iii \cap B(x,r) = \emptyset$ for all $\iii \in \Gamma(r)$ and $\#\Gamma(x,r) = 0$. On the other hand, if $r \ge \dist(x,E)$, then $B(x,r) \subset B(y,3r)$ and Lemma \ref{thm:ufcp_doubling} implies that there exists a constant $C \ge 1$ not depending on $x$ nor $r$ such that
  \begin{equation*}
    \#\Gamma(x,r) \le \#\{ \iii \in \Gamma(r) : E_\iii \cap B(y,3r) \ne \emptyset \} \le C.
  \end{equation*}
  The claim follows.
\end{proof}

\begin{example} \label{ex:ifs}
  \begin{figure}[!t]
    \subfigure{\includegraphics[width=5cm]{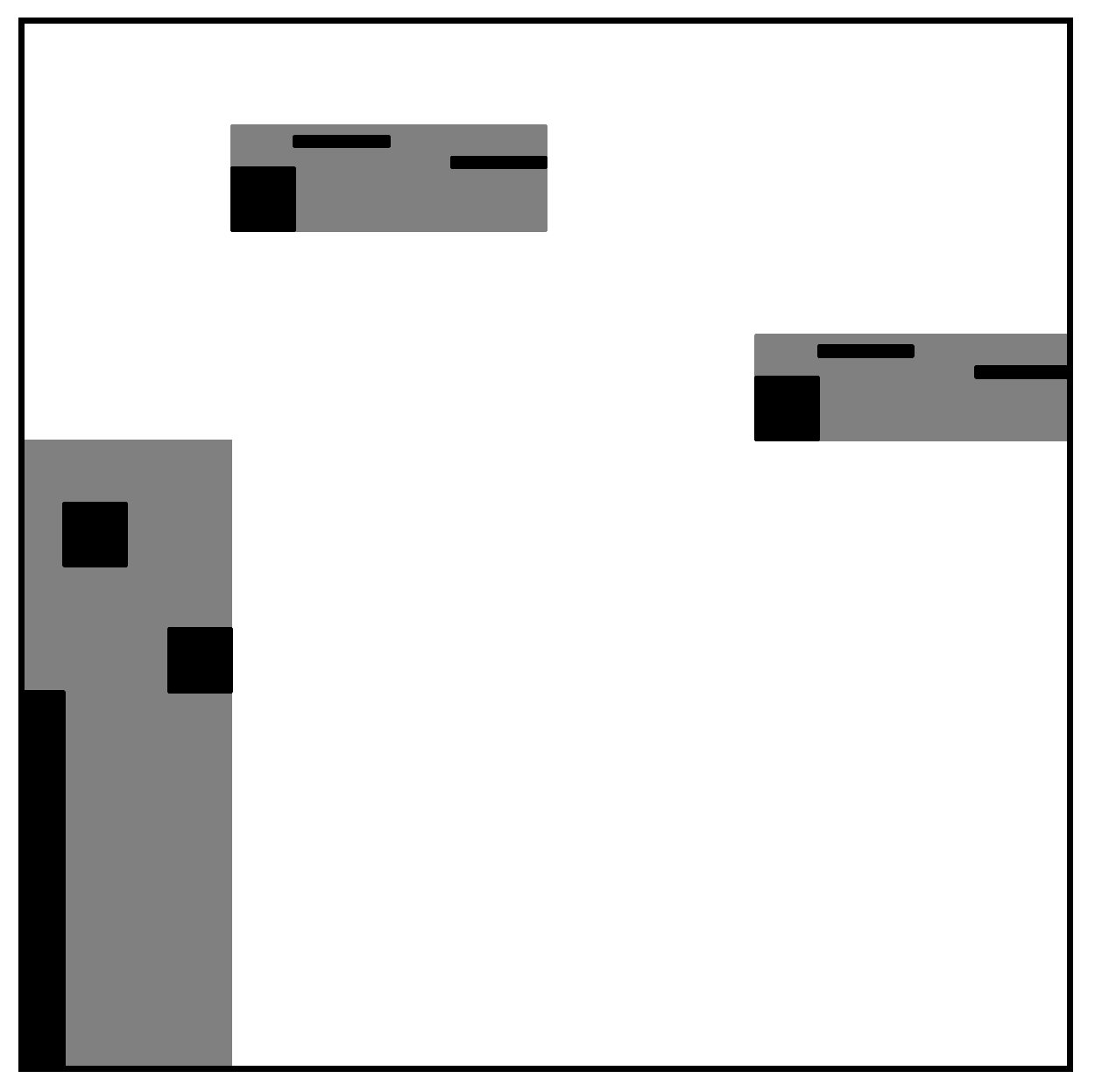}\label{fig:UFCP_yes}} \quad\quad\quad\quad
    \subfigure{\includegraphics[width=5cm]{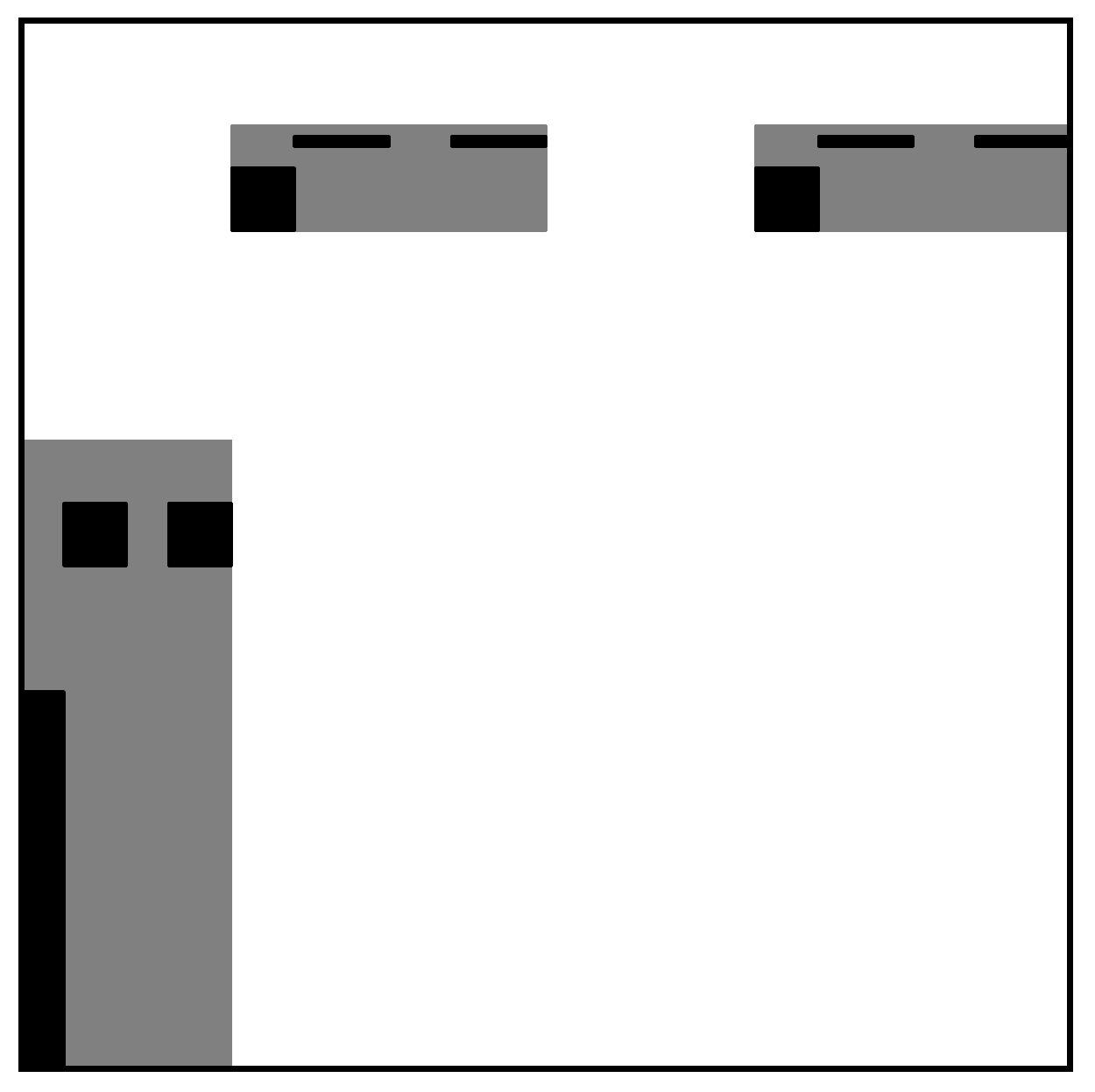}\label{fig:UFCP_no}}
    \caption{The self-affine set depicted in the left-hand side picture satisfies the uniform finite clustering property while the self-affine set of the right-hand side picture does not.} \label{fig:UFCP_yes_no}
  \end{figure}
  We exhibit a class of non-conformal iterated function systems satisfying the uniform finite clustering property. Write $Q = [0,1] \times [0,1]$ and let $\{ \fii_i \}_{i=1}^\kappa$ be an affine iterated function system acting on $Q$ so that for each $i \in \{ 1,\ldots,\kappa \}$ there are positive constants $r_i, s_i, a_i, b_i$ for which
  \begin{equation*}
    \fii_i(x,y) = (r_ix,s_iy) + (a_i,b_i).
  \end{equation*}
  It is easy to see that $\{ \fii_\iii(Q) : \iii \in \Sigma_* \}$ is a Moran construction satisfying \ref{M3} and \ref{M4}.

  We shall show that if
  \begin{align*}
    (a_i,a_i+r_i) \cap (a_j,a_j+r_j) &= \emptyset, \\
    (b_i,b_i+s_i) \cap (b_j,b_j+s_j) &= \emptyset
  \end{align*}
  whenever $i \ne j$, then the iterated function system $\{ \fii_i \}_{i=1}^\kappa$ satisfies the uniform finite clustering property; see Figure \ref{fig:UFCP_yes_no} for illustration. Since the linear parts of the mappings $\fii_i$ are diagonal we see that
  \begin{align}
    (a_\iii,a_\iii+r_\iii) \cap (a_\jjj,a_\jjj+r_\jjj) &= \emptyset, \label{eq:affine_UFCP1} \\
    (b_\iii,b_\iii+s_\iii) \cap (b_\jjj,b_\jjj+s_\jjj) &= \emptyset \label{eq:affine_UFCP2}
  \end{align}
  whenever $\iii \bot \jjj$.
  Here $r_\iii = r_{i_1} \cdots r_{i_n}$ and $a_\iii = \sum_{k=1}^{n} r_{\iii|_{k-1}}a_{i_k}$ for all $\iii=i_1\cdots i_n\in\Sigma_n$, and likewise for $s_\iii$ and $b_\iii$. Choose $M \in \N$ such that $\sqrt{2} < M\min\{ r_i,s_i : i \in \{ 1,\ldots,\kappa \} \}$. Fix $x \in Q$ and $r>0$. Let $\iii \in \Gamma(x,r)$ be such that $r_\iii \le s_\iii$. Since $s_\iii < \diam(\fii_\iii(Q)) \le \sqrt{2}s_\iii$ we have $r < \diam(\fii_{\iii^-}(Q)) \le Ms_\iii$. Observe that an interval of length $2r$ can intersect at most $2M+1$ mutually disjoint intervals of length $s_\iii$. Therefore, by \eqref{eq:affine_UFCP2}, we get
  \begin{equation*}
    \#\{ \iii \in \Gamma(x,r) : r_\iii \le s_\iii \} \le 2M+1.
  \end{equation*}
  By a symmetric argument, relying on \eqref{eq:affine_UFCP1}, we conclude that
  \begin{equation*}
    \#\Gamma(x,r) = \#\{ \iii \in \Gamma(x,r) : r_\iii \le s_\iii \} + \#\{ \iii \in \Gamma(x,r) : r_\iii > s_\iii \} \le 4M+2.
  \end{equation*}
  Hence the uniform finite clustering property is satisfied.
  
  Let us further show that if there are $i,j,k \in \{ 1,\ldots,\kappa \}$ such that $i \ne j$, $s_i=s_j$, $b_i=b_j$, and $r_k<s_k$, then the iterated function system $\{ \fii_i \}_{i=1}^\kappa$ does not satisfy the uniform finite clustering property.
  Notice that if $\iii, \jjj \in \{ i,j \}^n$, then $s_\iii = s_\jjj = s_i^n$ and $b_\iii = b_\jjj$. Since $r_k < s_k$ we find $m \in \N$ so that
  \begin{equation*}
    r_k^m < s_k^ms_i^n < \diam(\fii_k^m \circ \fii_\iii(Q))
  \end{equation*}
  for all $\iii \in \{ i,j \}^n$. Thus there is $x \in Q$ for which
  \begin{equation*}
    \fii_k^m \circ \fii_\iii(Q) \cap B(x,s_k^ms_i^n) \ne \emptyset
  \end{equation*}
  for all $\iii \in \{ i,j \}^n$ and, consequently,
  \begin{equation*}
    \# \Gamma(x,s_k^ms_i^n) \ge \#\{ i,j \}^n = 2^n
  \end{equation*}
  for all $n \ge n_0$. Lemma \ref{thm:sup_over_x} shows that the uniform finite clustering property cannot hold.
\end{example}

If a Moran construction $\{ E_\iii : \iii \in \Gamma_* \}$ satisfies \ref{M3} and \ref{M4}, then the \emph{topological pressure} $P \colon [0,\infty) \to \R$, defined by
\begin{equation*}
  P(t) = \lim_{n \to \infty} \tfrac{1}{n} \log\sum_{\iii \in \Gamma_n} \diam(E_\iii)^t,
\end{equation*}
is well-defined, convex, continuous, and strictly decreasing. The existence of the limit follows from \ref{M3}, convexity from H\"older's inequality, and for the right continuity at $0$, one needs \ref{M4}. Lemma \ref{upperalpha} guarantees that $P$ is strictly decreasing. Thus there exists unique $t \ge 0$ for which $P(t)=0$. For a more detailed reasoning, the reader is referred e.g.\ to \cite[Lemma 2.4]{RajalaVilppolainen2013} and \cite[Lemma 2.1]{KaenmakiVilppolainen2010}.

It is often the case that the dimension of the limit set is the zero of the topological pressure. The \emph{Assouad dimension} of a set $A \subset X$, denoted by $\dima(A)$, is the infimum of all $t$ satisfying the following: there exists a constant $C \ge 1$ such that each set $A \cap B(x,R)$ can be covered by at most $C(r/R)^{-t}$ balls of radius $r$ centered at $A$ for all $0<r<R$.

\begin{proposition} \label{thm:assouad_ylaraja}
  Suppose that $X$ is a complete metric space, $\Gamma \subset \Sigma$ is a compact set satisfying $\sigma(\Gamma) \subset \Gamma$, and $\{ E_\iii : \iii \in \Gamma_* \}$ is a Moran construction satisfying \ref{M3} and \ref{M4}. If the Moran construction satisfies the uniform finite clustering property, $E$ is its limit set, and $t \ge 0$ is so that $P(t) \le 0$, then $\dima(E) \le t$.
\end{proposition}

\begin{proof}
  We may assume that $P(t)<0$. Let $\delta>0$ be such that $P(t)<-\delta<0$ and choose $N \ge \delta^{-1}\log D^t$, where $D \ge 1$ is as in \ref{M3}, so that
  \begin{equation*}
    \sum_{\iii \in \Gamma_N} \diam(E_\iii)^t < e^{-N\delta} \le D^{-t}.
  \end{equation*}
  Observe that, by \ref{M3}, for each given $\iii \in \Gamma_{kN}$ we have
  \begin{equation*}
    \sum_{\iii\jjj \in \Gamma_{(k+1)N}} \diam(E_{\iii\jjj})^t \le D^t \diam(E_\iii)^t \sum_{\jjj \in \Gamma_N} \diam(E_\jjj)^t < \diam(E_\iii)^t
  \end{equation*}
  for all $k \in \N$. Therefore, if
  \begin{equation*}
    \Gamma_N(\roo) = \{ \iii \in \bigcup_{k=1}^\infty \Gamma_{kN} : \diam(E_\iii) < \roo \le \diam(E_{\iii|_{|\iii|-N}}) \},
  \end{equation*}
  then
  \begin{equation*}
    \sum_{\iii \in \Gamma_N(\roo)} \diam(E_\iii)^t < D^{-t} \le 1
  \end{equation*}
  for all $\roo > 0$. Since, by \ref{M4},
  \begin{equation*}
    1 > \sum_{\iii \in \Gamma_N(\roo)} \diam(E_\iii)^t \ge \lalpha^N \sum_{\iii \in \Gamma_N(\roo)} \diam(E_{\iii|_{|\iii|-N}})^t \ge \lalpha^N \sum_{\iii \in \Gamma_N(\roo)} \roo^t = \roo^t\lalpha^N \#\Gamma_N(\roo)
  \end{equation*}
  we have
  \begin{equation} \label{eq:assouad_ylaraja}
    \#\Gamma_N(\roo) \le \lalpha^{-N} \roo^{-t}
  \end{equation}
  for all $\roo>0$.
  
  According to the uniform finite clustering property, there is $M>0$ such that $\#\Gamma(x,R) \le M$ for all $x \in E$ and $R>0$. Fix $x \in E$ and $0<r<R$. Observe that the set $E \cap B(x,R)$ can be covered by the collection $\{ E_\hhh \}_{\hhh \in \Upsilon}$, where
  \begin{equation*}
    \Upsilon = \{ \iii\jjj \in \Gamma_* : \iii \in \Gamma(x,R) \text{ and } \jjj \in \Gamma_N(\tfrac{r}{DR}) \}.
  \end{equation*}
  Recalling \ref{M3}, we see that $\diam(E_{\iii\jjj}) \le D\diam(E_\iii)\diam(E_\jjj) \le r$ for all $\iii\jjj \in \Upsilon$. Thus the set $E \cap B(x,R)$ can be covered by $\#\Upsilon$ many balls of radius $r$ centered at $E$. Since, by \eqref{eq:assouad_ylaraja},
  \begin{equation*}
    \#\Upsilon \le \#\Gamma(x,R) \#\Gamma_N(\tfrac{r}{DR}) \le M\lalpha^{-N} \Bigl( \frac{r}{DR} \Bigr)^{-t}
  \end{equation*}
  we have finished the proof.
\end{proof}

Without the uniform finite clustering property, a simple modification of the previous proof yields an upper bound for the upper Minkowski dimension. This result is stated in \cite[Proposition 2.6]{RajalaVilppolainen2013} but it should be noted that the proof presented there is not correct. Recalling \cite[\S 3.1]{Fraser2014}, we see that, without the uniform finite clustering property, the result can fail for the Assouad dimension already in the self-similar case.

By combining Proposition \ref{thm:assouad_ylaraja} and \cite[Proposition 3.1]{RajalaVilppolainen2013}, we can now extend the dimension result for Moran constructions in metric spaces to cover also the Assouad dimension. We emphasize that this is possible due to the uniform finite clustering property assumption.
We also remark that \cite[Proposition 3.1]{RajalaVilppolainen2013} is stated in the case $\Gamma=\Sigma$, but it is easy to see that it is valid also in the general case. The $t$-dimensional Hausdorff measure is denoted by $\HH^t$ and the Hausdorff dimension of a set $A$ by $\dimh(A)$. Recall that $\dimh(A) \le \dima(A)$ for all sets $A$.

\begin{proposition} \label{thm:pressure_dim}
  Suppose that $X$ is a complete metric space, $\Gamma \subset \Sigma$ is a compact set satisfying $\sigma(\Gamma) \subset \Gamma$, and $\{ E_\iii : \iii \in \Gamma_* \}$ is a Moran construction satisfying \ref{M3} and \ref{M4}. If the Moran construction satisfies the uniform finite clustering property, $E$ is its limit set, and $t \ge 0$ is so that $P(t) = 0$, then $t = \dimh(E) = \dima(E)$ and $\HH^t(E)>0$.
\end{proposition}

\begin{remark} \label{rem:pressure}
  (1) A limit set $E$ of a Moran construction satisfying the assumptions of Proposition \ref{thm:pressure_dim} can have $\HH^t(E)=\infty$ for $t \ge 0$ with $P(t)=0$; see \cite[Example 6.4]{KaenmakiVilppolainen2010} and \cite[Example 2.3]{RajalaVilppolainen2013}.
  
  (2) For a sub-self-conformal set $E$ we have $\PP^t(E)<\infty$ for $t = \dimh(E)$ even without assuming the finite clustering property. Here $\PP^t$ is the $t$-dimensional packing measure. This follows by applying the argument used in \cite[Theorem 4.3]{KaenmakiVilppolainen2008} and \cite[Theorem 3.2]{LauNgaiWang2009}.
\end{remark}

\subsection{Weak separation condition}
Suppose that $\{ \fii_i \}_{i=1}^\kappa$ be an iterated function system on a complete metric space $X$. Let $\{ \fii_\iii(W) : \iii \in \Sigma_* \}$ be a Moran construction as in Lemma \ref{thm:ifs} and $E$ its limit set. We define
\begin{equation*}
  \Phi(r) = \{ \fii_\iii : \iii \in \Sigma(r) \}
\end{equation*}
for all $r>0$ and
\begin{equation*}
  \Phi(x,r) = \{ \fii \in \Phi(r) : \fii(W) \cap B(x,r) \ne \emptyset \}
\end{equation*}
for all $x \in X$ and $r>0$. We say that the iterated function system $\{ \fii_i \}_{i=1}^\kappa$ satisfies the \emph{weak separation condition} if
\begin{equation*}
  \sup_{x \in E} \limsup_{r \downarrow 0} \# \Phi(x,r) < \infty
\end{equation*}
and the \emph{uniform weak separation condition} if
\begin{equation*}
  \sup_{x \in E} \sup_{r>0} \# \Phi(x,r) < \infty.
\end{equation*}
Note that $\#\Phi(x,r) \le \#\Sigma(x,r)$ for all $x \in X$ and $r>0$. Furthermore, if $\fii_\iii \ne \fii_\jjj$ for all $\iii,\jjj \in \Sigma_*$ with $\iii \ne \jjj$, then $\#\Phi(x,r) = \#\Sigma(x,r)$ for all $x \in X$ and $r>0$.

\begin{remark} \label{rem:wsc}
  (1) The weak separation condition originates in \cite{LauNgai1999, Zerner1996}. It was first defined for self-similar sets in Euclidean spaces. The definition for self-conformal sets was introduced in \cite{LauNgaiWang2009}. There it was assumed that there exists a set $D \subset \R^d$ with non-empty interior so that
  \begin{equation*}
    \sup_{x \in X} \sup_{r>0} \# \{ \fii \in \Phi(r) : x \in \fii(D) \} < \infty.
  \end{equation*}
  Choosing the set $D$ to be the set $W$ of Lemma \ref{thm:ifs}, it follows that this condition is equivalent to the uniform weak separation condition; recall Lemma \ref{thm:sup_over_x} and inspect the proof of the implication (a) $\Rightarrow$ (b) in \cite[Proposition 3.1]{LauNgaiWang2009}.

  (2) It follows immediately from the definitions that an iterated function system $\{ \fii_i \}_{i=1}^\kappa$ satisfies the finite clustering property if and only if it satisfies the weak separation condition and $\fii_\iii \ne \fii_\jjj$ for all $\iii,\jjj \in \Sigma_*$ with $\iii \ne \jjj$. Recalling that for a self-conformal set, the finite clustering property is equivalent to the open set condition, we have obtained an alternative proof for \cite[Theorem 1.3]{DengNgai2011}.
\end{remark}

The following proposition shows that the weak separation condition is always related to the finite clustering property.

\begin{proposition} \label{thm:wsc}
  Suppose that $X$ is a complete metric space, $\{ \fii_i \}_{i=1}^\kappa$ is an iterated function system satisfying the (uniform) weak separation condition, $W \subset X$ is as in Lemma \ref{thm:ifs}, and $E$ is the limit set of the Moran construction $\{ \fii_\iii(W) : \iii \in \Sigma_* \}$. Then there exists a compact set $\Gamma \subset \Sigma$ satisfying $\sigma(\Gamma) \subset \Gamma$ such that
  \begin{equation*}
    \#\Gamma(x,r) = \#\Phi(x,r)
  \end{equation*}
  for all $x \in E$ and $r>0$. In particular, the Moran construction $\{ \fii_\iii(W) : \iii \in \Gamma_* \}$ satisfies the (uniform) finite clustering property and has $E$ as its limit set.
\end{proposition}

\begin{proof}
  We will define a compact set $\Gamma \subset \Sigma$ satisfying $\sigma(\Gamma) \subset \Gamma$ so that
  \begin{enumerate}
    \item for each $\iii \in \Sigma_*$ there is $\jjj \in \Gamma_*$ such that
$\fii_\iii = \fii_\jjj$,
    \item $\fii_\iii \ne \fii_\jjj$ for all $\iii, \jjj \in \Gamma_*$ with $\iii
\ne \jjj$.
  \end{enumerate}
  Let $\prec$ be the lexicographical order on $\Sigma_*$. This means that $\iii \prec \jjj$ if $\iii = \iii \wedge \jjj$ or $\iii = (\iii \wedge \jjj) i \iii'$ and $\jjj = (\iii \wedge \jjj) j \jjj'$, where $i,j \in \Sigma_1$ so that $i<j$ and $\iii',\jjj' \in \Sigma_*$. Using this order we define a set
  \begin{equation*}
    R = \{ \kkk \in \Sigma_* : \fii_\kkk = \fii_\hhh \text{ for some } \hhh \in \Sigma_* \setminus \{ \kkk \} \text{ with } \hhh \prec \kkk \}.
  \end{equation*}
  It follows immediately that $\fii_\iii \ne \fii_\jjj$ for all $\iii, \jjj \in \Sigma_* \setminus R$ with $\iii \ne \jjj$. Observe that if $\kkk \in R$ and $\iii, \jjj \in \Sigma_*$, then $\iii\kkk\jjj \in R$. Indeed, if $\hhh \in \Sigma_*$ is such that $\hhh \prec \kkk$ and $\fii_\kkk = \fii_\hhh$, then $\iii\hhh\jjj \prec \iii\kkk\jjj$ and $\fii_{\iii\hhh\jjj} = \fii_{\iii\kkk\jjj}$. Thus, defining $\Gamma = \Sigma[R]$, we have found a compact set $\Gamma$ for which $\sigma(\Gamma) \subset \Gamma$. In fact, we have proven both (1) and (2) since the above reasoning shows that $\Gamma_* = \Sigma[R]_* = \Sigma_* \setminus R$.

  The (uniform) finite clustering property follows now immediately since $\Phi(x,r) = \{ \fii_\iii : \iii \in \Gamma(r)$ and $\fii_\iii(W) \cap B(x,r) \ne \emptyset \}$ by (1), and (2) implies $\#\{ \fii_\iii : \iii \in \Gamma(r)$ and $\fii_\iii(W) \cap B(x,r) \ne \emptyset \} = \#\Gamma(x,r)$ for all $x \in X$ and $r>0$. The claim on the limit set follows from (1).
\end{proof}

\begin{remark}
  Suppose that $\{ \fii_i \}_{i=1}^\kappa$ is a bi-Lipschitz iterated function system satifying \eqref{BDP_ehto} and $E$ its invariant set. If $\{ \fii_i \}_{i=1}^\kappa$ satisfies the uniform weak separation condition, then Lemma \ref{thm:ifs} and Propositions \ref{thm:pressure_dim} and \ref{thm:wsc} imply that $t = \dimh(E) = \dima(E)$ and $\HH^t(E)>0$, where $t$ is the zero of the corresponding pressure. Furthermore, if $\{ \fii_i \}_{i=1}^\kappa$ is conformal and satisfies the weak separation condition, then Remark \ref{rem:pressure}(2) gives $\PP^t(E) < \infty$.
\end{remark}

\subsection{Lower local dimension}
Suppose that $X$ is a locally compact metric space. If $\mu$ is a locally finite Borel regular measure on $X$ and $x \in X$, then the \emph{lower local dimension} of $\mu$ at $x$ is defined by
\begin{equation*}
  \ldimloc(\mu,x) = \liminf_{r\downarrow 0}\frac{\log\mu(B(x,r))}{\log r}.
\end{equation*}
The \emph{lower Hausdorff dimension} of $\mu$ is defined by
\begin{equation*}
  \ldimh(\mu) = \sup\{ t \ge 0 : \ldimloc(\mu,x) \ge t \text{ for $\mu$-almost all $x$} \}.
\end{equation*}
Recall that this quantity can be recovered from the set-theoretical Hausdorff dimension as follows:
\begin{equation*}
  \ldimh(\mu) = \inf\{ \dimh(A) : A \subset X \text{ is a Borel with } \mu(A)>0 \}.
\end{equation*}
The reader is referred to the books of Mattila \cite{Mattila1995} and Falconer \cite{Falconer1997} for further background on measures and dimensions.

The next proposition shows that under the uniform finite clustering property, the lower local dimension of the projected measure can be obtained symbolically. It generalizes \cite[Proposition 3.1]{KaenmakiRajalaSuomala2012b}. If $f$ is a function and $\mu$ is a measure, then the push-forward measure is denoted by $f\mu$.

\begin{proposition} \label{thm:dimloc}
  Suppose that $X$ is a complete metric space, $\Gamma \subset \Sigma$ is a compact set satisfying $\sigma(\Gamma) \subset \Gamma$, and $\mu$ is a Borel regular probability measure on $\Gamma$. If $\{ E_\iii : \iii \in \Gamma \}$ is a Moran construction satisfying \ref{M3} and the uniform finite clustering property, then
  \begin{equation*}
    \ldimloc(\pi\mu, \pi(\iii)) = \liminf_{n \to \infty} \frac{\log\mu([\iii|_n])}{\log\diam(E_{\iii|_n})}
  \end{equation*}
  for $\mu$-almost all $\iii \in \Gamma$.
\end{proposition}

\begin{proof}
  We may clearly assume that $\mu$ has no atoms. Fix $\iii \in \Gamma$. Since $E_{\iii|_n} \subset B(\pi(\iii), \diam(E_{\iii|_n}))$ and $\mu([\iii|_n]) \le \pi\mu(E_{\iii|_n})$ for all $n \in \N$ we have
  \begin{equation*}
    \ldimloc(\pi\mu,\pi(\iii)) \le \liminf_{n \to \infty} \frac{\log\pi\mu(B(\pi(\iii), \diam(E_{\iii|_n})))}{\log\diam(E_{\iii|_n})}
    \le \liminf_{n \to \infty} \frac{\log\mu([\iii|_n])}{\log\diam(E_{\iii|_n})}.
  \end{equation*}
  To show the other inequality, we first define
  \begin{equation*}
    T(\iii,r) = \{ \jjj \in \Gamma(r) : \dist(E_\iii,E_\jjj) \le r \}
  \end{equation*}
  for all $\iii \in \Gamma_*$ and $r>0$. According to Lemma \ref{thm:ufcp_doubling}, there is a constant $C \ge 1$ such that
  \begin{equation*}
    \# T(\iii,r) \le \#\{ \jjj \in \Gamma(r) : E_\jjj \cap B(x,2r) \ne \emptyset \} \le C
  \end{equation*}
  for all $r>0$, $\iii \in \Gamma(r)$, and $x \in \pi([\iii])$. Observe that also $\#\{ \jjj \in \Gamma(r) : \iii \in T(\jjj,r) \} \le C$ for all $\iii \in \Gamma(r)$ since otherwise $\# T(\iii,r) > C$ for some $\iii \in \Gamma(r)$. Thus each $\iii \in \Gamma(r)$ is contained in at most $C$ sets $T(\jjj,r)$ where $\jjj \in \Gamma(r)$. Therefore
  \begin{equation} \label{eq:summasumma}
    \sum_{\iii \in \Gamma(r)} \sum_{\jjj \in T(\iii,r)} \mu([\jjj]) \le \sum_{\iii \in \Gamma(r)} C\mu([\iii]) = C
  \end{equation}
  for all $r>0$.
  
  Let $\eps>0$ and define
  \begin{equation*}
    A_\eps(r) = \{ \iii \in \Gamma(r) : \mu([\iii]) \ge r^\eps \sum_{\jjj \in T(\iii,r)} \mu([\jjj]) \}
  \end{equation*}
  for all $r>0$. Recalling \eqref{eq:summasumma}, we have
  \begin{equation*}
    \mu\biggl( \Gamma \setminus \bigcup_{\iii \in A_\eps(r)} [\iii] \biggr) = \sum_{\iii \in \Gamma(r) \setminus A_\eps(r)} \mu([\iii]) \le r^\eps \sum_{\iii \in \Gamma(r) \setminus A_\eps(r)} \sum_{\jjj \in T(\iii,r)} \mu([\jjj]) \le Cr^\eps
  \end{equation*}
  for all $r>0$. If $0<\gamma<1$, then $\sum_{n=1}^\infty \mu(\Gamma \setminus \bigcup_{\iii \in A_\eps(\gamma^n)} [\iii]) < \infty$ and the Borel-Cantelli Lemma gives $\mu(A_\eps) = 1$ for $A_\eps = \bigcup_{N=1}^\infty \bigcap_{n=N}^\infty \bigcup_{\iii \in A_\eps(\gamma^n)} [\iii]$. Fix $\hhh \in A_\eps$ and notice that there exists $N \in \N$ such that $\hhh \in \bigcup_{\iii \in A_\eps(\gamma^n)} [\iii]$ for all $n \ge N$. Let $0<r<\gamma^N$ and choose $n \ge N$ so that $\gamma^{n+1} \le r < \gamma^n$. Let $k \in \N$ be such that $\hhh|_k \in \Gamma(\gamma^n)$. Notice that if $r \downarrow 0$, then $n \to \infty$ and $k \to \infty$. Since $\pi^{-1}(B(\pi(\hhh),\gamma^n)) \subset \bigcup_{\jjj \in T(\hhh|_k,\gamma^n)} [\jjj]$ and $\gamma^n \ge \diam(E_{\hhh|_k})$ we have
  \begin{equation} \label{eq:dimloc_alaraja}
  \begin{split}
    \frac{\log \pi\mu(B(\pi(\hhh),r))}{\log r} &\ge \frac{\log \pi\mu(B(\pi(\hhh),\gamma^n))}{\log \gamma^{n+1}}
    \ge \frac{\log\sum_{\jjj \in T(\hhh|_k,\gamma^n)} \mu([\jjj])}{\log \gamma^{n+1}} \\
    &\ge \frac{\log \gamma^{-n\eps} \mu([\hhh|_k])}{\log\gamma^{n+1}}
    \ge -\eps\frac{n}{n+1} + \frac{\log\mu([\hhh|_k])}{\log\gamma + \log\diam(E_{\hhh|_k})}
  \end{split}
  \end{equation}
  for all $r>0$. Letting $r \downarrow 0$ and then $\eps \downarrow 0$ gives the desired lower bound.
\end{proof}

\subsection{Microsets} \label{sec:microset}
To finish this section, we will define microsets in the shift space and show how they can be used to calculate the Assouad dimension of any closed subset of a self-conformal set.

Let $\Gamma \subset \Sigma$ be a compact set satisfying $\sigma(\Gamma) \subset \Gamma$. We define $\beta_\iii(A) = \sigma^{|\iii|}(A \cap [\iii])$ for all $A \subset \Gamma$ and $\iii \in \Sigma_*$. Note that if $A \subset \Gamma$, then $A \cap [\iii] \subset \Gamma$ and $\beta_\iii(A) \subset \sigma^{|\iii|}(\Gamma) \subset \Gamma$ for all  $\iii \in \Sigma_*$. Since $\beta_\iii(\{ \hhh \}) = \sigma^{|\iii|}(\hhh)$ and $d(\sigma^{|\iii|}(\hhh), \sigma^{|\iii|}(\kkk)) = 2^{|\iii|} d(\hhh,\kkk)$ for all $\hhh, \kkk \in \Gamma \cap [\iii]$, where $d$ is the metric of $\Sigma$, we see that, as a function defined on the family of all compact subsets of $\Gamma$, $\beta_\iii$ is continuous with respect to the Hausdorff metric.

We say that $A' \subset \Gamma$ is a \emph{miniset} of $A \subset \Gamma$ if $A' = \beta_\iii(A)$ for some $\iii \in \Sigma_*$. Furthermore, $A' \subset \Gamma$ is a \emph{microset} of a compact set $A \subset \Gamma$ if there exists a sequence $(A_n')_{n=1}^\infty$ of minisets of $A$ such that $A_n' \to A'$ in the Hausdorff metric. Note that a miniset of a compact set is clearly a microset.

The following lemma shows that microsets enjoy additional spatial invariance analogous to the well-known Preiss' principle that ``tangent measures to tangent measures are tangent measures''.

\begin{lemma} \label{thm:micromicro}
  Suppose that $\Gamma \subset \Sigma$ is a compact set satisfying $\sigma(\Gamma) \subset \Gamma$. If $A' \subset \Gamma$ is a microset of a compact set $A \subset \Gamma$ and $A'' \subset \Gamma$ is a microset of $A'$, then $A''$ is a microset of $A$.
\end{lemma}

\begin{proof}
  Let $(\beta_{\iii_n}(A))_{n=1}^\infty$ be a sequence of minisets converging to $A'$ in the Hausdorff metric and $(\beta_{\jjj_m}(A'))_{m=1}^\infty$ a sequence converging to $A''$. Since $\beta_{\iii_n}(A) \to A'$ we have $\beta_{\jjj_m}(\beta_{\iii_n}(A)) \to \beta_{\jjj_m}(A')$ for all $m \in \N$. Observing that
  \begin{equation*}
    \beta_{\jjj_m}(\beta_{\iii_n}(A))
    = \sigma^{|\jjj_m|}(\sigma^{|\iii_n|}(A \cap [\iii_n]) \cap [\jjj_m])
    = \sigma^{|\jjj_m|}(\sigma^{|\iii_n|}(A \cap [\iii_n] \cap [\iii_n\jjj_m]))
    = \beta_{\iii_n\jjj_m}(A)
  \end{equation*}
  for all $n,m \in \N$, we see that $(\beta_{\iii_n\jjj_m}(A))_{n=1}^\infty$ is a sequence of minisets of $A$ converging to $\beta_{\jjj_m}(A')$.
  
  Let $\eps>0$. Since $\beta_{\jjj_m}(A') \to A''$ there exists $m_0 \in \N$ such that $D(\beta_{\jjj_m}(A'),A'') < \eps/2$ for all $m \ge m_0$. Here $D$ is the Hausdorff metric. Furthermore, since $\beta_{\iii_n\jjj_m}(A) \to \beta_{\jjj_m}(A')$ for all $m \in \N$ we see that for each $m \ge m_0$ there is $n(m) \in \N$ such that $D(\beta_{\iii_{n(m)}\jjj_m}(A),\beta_{\jjj_m}(A')) < \eps/2$. Thus for each $m \ge m_0$ we have
  \begin{equation*}
    D(\beta_{\iii_{n(m)}\jjj_m}(A),A'') \le D(\beta_{\iii_{n(m)}\jjj_m}(A),\beta_{\jjj_m}(A')) + D(\beta_{\jjj_m}(A'),A'') < \eps.
  \end{equation*}
  This shows that $(\beta_{\iii_{n(m)}\jjj_m}(A))_{m=1}^\infty$ is a sequence of minisets of $A$ converging to $A''$ and hence, $A''$ is a microset of $A$.
\end{proof}

If $A \subset \Gamma$ is compact, then we define
\begin{equation*}
  N_n(A) = \max\{ \#\{ \iii \in \Gamma_n : A' \cap [\iii] \ne \emptyset \} : A' \text{ is a microset of } A \}
\end{equation*}
for all $n \in \N$.

\begin{lemma} \label{thm:submulti}
  If $\Gamma \subset \Sigma$ is a compact set satisfying $\sigma(\Gamma) \subset \Gamma$ and $A \subset \Gamma$ is compact, then the sequence $(N_n(A))_{n=1}^\infty$ is sub-multiplicative.
\end{lemma}

\begin{proof}
  Fix $n,m \in \N$ and let $A' \subset \Gamma$ be a microset of $A$ such that
  \begin{equation*}
    N_{n+m}(A) = \#\{ \hhh \in \Gamma_{n+m} : A' \cap [\hhh] \ne \emptyset \}.
  \end{equation*}
  Fix $\hhh \in \Gamma_{n+m}$ such that $A' \cap [\hhh] \ne \emptyset$ and let $\iii = \hhh|_n \in \Gamma_n$ and $\jjj = \sigma^n(\hhh) \in \Gamma_m$. Since $A' \cap [\iii\jjj] \ne \emptyset$ we trivially have $A' \cap [\iii] \ne \emptyset$ and $\iii \in \{ \kkk \in \Gamma_n : A' \cap [\kkk] \ne \emptyset \}$. Observe that we also have
  \begin{equation*}
    \emptyset \ne \beta_\iii(A' \cap [\iii\jjj]) = \sigma^{|\iii|}(A' \cap [\iii\jjj] \cap [\iii]) = \sigma^{|\iii|}(A' \cap [\iii]) \cap [\jjj] = \beta_\iii(A') \cap [\jjj]
  \end{equation*}
  and thus $\jjj \in \{ \kkk \in \Gamma_m : \beta_\iii(A') \cap [\kkk] \ne \emptyset \}$. Since, by Lemma \ref{thm:micromicro}, $\beta_\iii(A')$ is a microset of $A$ we conclude that
  \begin{equation*}
    N_{n+m}(A) \le \#\{ \kkk \in \Gamma_n : A' \cap [\kkk] \ne \emptyset \} \cdot N_m(A) \le N_n(A) N_m(A)
  \end{equation*}
  for all $n,m \in \N$.
\end{proof}

The following proposition relates the asymptotic behavior of microsets to the Assouad dimension.

\begin{proposition} \label{thm:dimf}
  Suppose that $X$ is a complete metric space and $\Gamma \subset \Sigma$ is a compact set satisfying $\sigma(\Gamma) \subset \Gamma$. Let $\{ E_\iii : \iii \in \Gamma_* \}$ be a Moran construction satisfying \ref{M3} and the uniform finite clustering property, and let $0<\ualpha<1$ be as in Lemma \ref{upperalpha}. If $A \subset \Gamma$ is compact, then
  \begin{equation*}
    \dima(\pi(A')) \le \lim_{n \to \infty} \frac{\log N_n(A)}{\log\ualpha^{-n}} 
  \end{equation*}
  for all microsets $A'$ of $A$.
\end{proposition}

\begin{proof}
  Observe first that the limit on the right-hand side exists by Lemma \ref{thm:submulti} and the standard theory of sub-multiplicative sequences. Fix $t > \lim_{n \to \infty} \log N_n(A)/\log\ualpha^{-n}$ and let $A' \subset \Gamma$ be a microset of $A$. Let $0<R<\diam(\pi(A'))$ and $x \in \pi(A')$. The uniform finite clustering property implies that there exists a constant $M \in \N$ not depending on $x$ nor $R$ such that $\#\Gamma(x,R) \le M$. Observe that if $\iii \in \Gamma(R)$ is such that $\pi([\iii]) \cap B(x,R) \ne \emptyset$, then $\iii \in \Gamma(x,R)$. Thus it suffices to show that for each $\iii \in \Gamma(x,R)$ the set $\pi(A' \cap [\iii])$ can be covered by at most $C_0(r/R)^{-t}$ balls of radius $r$ for all $0<r<R$.
  
  By the choice of $t$ there is $n_0 \in \N$ such that $N_n(A) \le \ualpha^{-nt}$ for all $n \ge n_0$. Fix $\iii \in \Gamma(x,R)$ and $0<r<R$. Choose $n \ge n_0$ so that
  \begin{equation} \label{eq:n0choice}
    \ualpha^{n-n_0} \le \frac{r}{CDR} < \ualpha^{n-n_0-1}
  \end{equation} 
  where $D \ge 1$ is as in \ref{M3} and $C \ge 1$ is as in Lemma \ref{upperalpha}. Recall that, by Lemma \ref{thm:micromicro}, $\beta_\iii(A' \cap [\iii]) = \beta_\iii(A')$ is a microset of $A$. Denoting $N_n(\iii) = \{ \jjj \in \Gamma_n : \beta_\iii(A') \cap [\jjj] \ne \emptyset \}$, the choices of $t$ and $n$ give
  \begin{equation} \label{eq:af1}
    \# N_n(\iii) \le N_n(A) \le \ualpha^{-nt} < \ualpha^{-(n_0+1)t} C^tD^t \Bigl( \frac{r}{R} \Bigr)^{-t}.
  \end{equation} 
  Observe that $A' \cap [\iii\jjj] \ne \emptyset$ and thus $\iii\jjj \in \Gamma_*$ for all $\jjj \in N_n(\iii)$. It follows that $A' \cap [\iii] \subset \bigcup_{\jjj \in N_n(\iii)} [\iii\jjj]$ and hence
  \begin{equation} \label{eq:af2}
    \pi(A' \cap [\iii]) \subset \bigcup_{\jjj \in N_n(\iii)} E_{\iii\jjj}.
  \end{equation} 
  If $\jjj \in N_n(\iii)$, then, by \ref{M3}, the fact that $\iii \in \Gamma(R)$, Lemma \ref{upperalpha}, and \eqref{eq:n0choice}, we get
  \begin{equation} \label{eq:af3}
    \diam(E_{\iii\jjj}) \le D\diam(E_\iii)\diam(E_\jjj) \le CDR\ualpha^n \le r.
  \end{equation} 
  The proof is now finished since \eqref{eq:af2}, \eqref{eq:af3}, and \eqref{eq:af1} show that $\pi(A' \cap [\iii])$ can be covered by at most $C_0(r/R)^{-t}$ balls of radius $r$.
\end{proof}

\begin{remark}
  If $\Gamma \subset \Sigma$ is a compact set satisfying $\sigma(\Gamma) \subset \Gamma$ and $\{ E_\iii : \iii \in \Gamma_* \}$ is a Moran construction on a complete metric space, then
  \begin{equation*}
    \dima(\pi(A)) \le \sup\{ \dima(\pi(A')) : A' \text{ is a microset of } A \}
  \end{equation*}
  for all compact sets $A \subset \Gamma$. This is a triviality since $A$ is a microset of $A$. It is also easy to come up with a Moran construction
  for which the above inequality is strict for some set $A$. For example, if $\pi([j])$ is a square and the projection of $A = [ij]$ is a line, then $\beta_i(A) = [j]$ and $\dima(\pi(A)) = 1 < 2 = \dima(\pi(\beta_i(A)))$. However, if $E$ is the limit set of the Moran construction, then $\pi(A') \subset E$ for all microsets $A'$ of $\Gamma$ and
  \begin{equation*}
    \dima(E) = \sup\{ \dima(\pi(A')) : A' \text{ is a microset of } \Gamma \}
  \end{equation*}
  by the monotonicity of the Assouad dimension.
\end{remark}

If the magnification $\beta_\iii$, considered as an action on the complete metric space $X$, is geometrically nice enough, we are able to express the Assouad dimension of a subset of the limit set by means of its microsets.

\begin{proposition} \label{thm:dima_cifs}
  Suppose that $\{ \fii_i \}_{i=1}^\kappa$ is a conformal iterated function system, $\Gamma \subset \Sigma$ is a compact set satisfying $\sigma(\Gamma) \subset \Gamma$, and $\{ \fii_\iii(W) : \iii \in \Gamma_* \}$ is a Moran construction as in Lemma \ref{thm:ifs}. If $A \subset \Gamma$ is compact, then
  \begin{equation*}
    \dima(\pi(A)) = \sup\{ \dima(\pi(A')) : A' \text{ is a microset of } A \}.
  \end{equation*}
\end{proposition}

\begin{proof}
  Fix $t > \dima(\pi(A))$ and let $A'$ be a miniset of $A$. By showing that there exists a constant $C_0 \ge 1$ not depending on $A'$ so that $\pi(A') \cap B(x,R)$ can be covered by at most $C_0(r/R)^{-t}$ balls of radius $r$ for all $0<r<R<\diam(\pi(A'))$, we have proven the claim.
  
  Let $\iii \in \Sigma_*$ be such that $A' = \beta_\iii(A)$. Notice that $\pi(A') = \fii_\iii^{-1}(\pi(A \cap [\iii]))$. According to the well known bounded distortion property (see e.g.\ \cite[\S 2]{MauldinUrbanski1996}), there exists $C_1 \ge 1$ such that
  \begin{equation} \label{eq:tupu}
    B(\fii_\iii(x),C_1^{-1}\diam(\fii_\iii(W))r) \subset \fii_\iii(B(x,r)) \subset B(\fii_\iii(x),C_1\diam(\fii_\iii(W))r)
  \end{equation}
  for all $x \in W$ and $r>0$ with $B(x,r) \subset W$. Fix $0<r<R<\diam(\pi(A'))$ and observe that
  \begin{equation} \label{eq:hupu}
  \begin{split}
    \pi(A') \cap B(x,R) &= \fii_\iii^{-1}\bigl( \pi(A \cap [\iii]) \cap \fii_\iii(B(x,R)) \bigr) \\
    &\subset \fii_\iii^{-1}\bigl( \pi(A \cap [\iii]) \cap B(\fii_\iii(x),C_1\diam(\fii_\iii(W))R) \bigr).
  \end{split}
  \end{equation}
  By the choice of $t$, there exists $C_2 \ge 1$ not depending on $A'$ such that the set $\pi(A \cap [\iii]) \cap B(\fii_\iii(x),C_1\diam(\fii_\iii(W))R)$ can be covered by at most $C_2(r/R)^{-t}$ balls of radius $C_1\diam(\fii_\iii(W))r$ centered at $\pi(A \cap [\iii])$. Thus, by \eqref{eq:tupu} and \eqref{eq:hupu}, the set $\pi(A') \cap B(x,R)$ can be covered by at most $C_2(r/R)^{-t}$ balls of radius $C_1^2 r$ centered at $\pi(A')$. Since the set $\pi(A')$, as a subset of $\R^d$, is doubling there exists a constant $N \in \N$ depending only on $d$ and $C_1$ so that any ball of radius $C_1^2 r$ can be covered by at most $N$ balls of radius $r$. The claim follows.
\end{proof}

\section{Distributions} \label{sec:distributions}

In this section, we review the main properties of CP-distributions. We aim to present only the properties we need. We will be brief in places which rely on standard measure theoretical arguments.

\subsection{Adapted distributions}
Suppose that $\Gamma \subset \Sigma$ is a compact set satisfying $\sigma(\Gamma) \subset \Gamma$. Let $\PP(X)$ be the set of all Borel probability measures defined on a given metric space $X$. Recall that if $X$ is compact, then $\PP(X)$ is metrizable and compact in the weak topology.
We consider the set $\PP(\Sigma)$ and let $\PP(\Gamma) = \{ \mu \in \PP(\Sigma) : \spt(\mu) \subset \Gamma \}$. We define
\begin{equation} \label{eq:omegagamma}
  \Omega_\Gamma = \{ (\mu,\iii) \in \PP(\Gamma) \times \Gamma : \iii \in \spt(\mu) \}.
\end{equation}
The space $\Omega_\Gamma$ has the subspace topology inherited from the product space $\PP(\Gamma) \times \Gamma$. Since $(\mu,\iii) \mapsto \mu([\iii|_n])$ is continuous we see that the set $\{ (\mu,\iii) \in \PP_\Gamma : \mu([\iii|_n]) > 0 \}$ is open for all $n \in \N$ and therefore
\begin{equation*}
  \Omega_\Gamma = \bigcap_{n=0}^\infty \{ (\mu,\iii) \in \PP_\Gamma : \mu([\iii|_n]) > 0 \}
\end{equation*}
is a Borel set.

Measures on $\PP(\Gamma)$ and $\PP(\Gamma) \times \Gamma$ will be called distributions. To simplify notation, we abbreviate $\PP(\Gamma) \times \Gamma$ as $\PP_\Gamma$.
We say that a distribution $Q$ on $\PP_\Gamma$ is \emph{adapted} if there exists a distribution $\overline{Q}$ on $\PP(\Gamma)$ such that
\begin{equation} \label{eq:adapted}
  \int_{\PP_\Gamma} f(\mu,\iii) \dd Q(\mu,\iii) = \int_{\PP(\Gamma)} \int_{\Gamma} f(\mu,\iii) \dd\mu(\iii) \dd\overline{Q}(\mu)
\end{equation}
for all $f \in C(\PP_\Gamma)$. Here $C(X)$ is the space of all continuous functions $X \to \R$. In other words, $Q$ is adapted if choosing a pair $(\mu,\iii)$ according to $Q$ can be done in two-step process, by first choosing $\mu$ according to $\overline{Q}$ and then choosing $\iii$ according to $\mu$.

The following five lemmas summarize the properties of adapted distributions. The goal is to verify that limits of certain averages of adapted distributions remain adapted. This is crucial in the proof of Proposition \ref{thm:exists-invariant}.

\begin{lemma} \label{thm:adapted_def}
  A distribution $Q$ on $\PP_\Gamma$ is adapted if and only if there exists a distribution $\overline{Q}$ on $\PP(\Gamma)$ such that
  \begin{equation*}
    \int_{\PP_\Gamma} f(\mu,\iii) \dd Q(\mu,\iii) = \int_{\PP(\Gamma)} \int_{\Gamma} f(\mu,\iii) \dd\mu(\iii) \dd\overline{Q}(\mu)
  \end{equation*}
  for all essentially bounded and measurable functions $f \colon \PP_\Gamma \to \R$.
\end{lemma}

\begin{proof}
  Since the other direction is trivial let us assume that $Q$ is adapted. The first step is to show that the claim holds for the indicator function of an open set. This follows from the dominated convergence theorem and \eqref{eq:adapted} since, by Urysohn's lemma, we may approximate the indicator function by a continuous function arbitrary well. The proof of the claim then follows from Lusin's theorem.
\end{proof}

The following lemma follows from \cite[Lemma 5.6]{Hochman2010}. Our proof below exhibits an alternative argument.

\begin{lemma} \label{thm:adapted_omega}
  If $Q$ is an adapted distribution on $\PP_\Gamma$, then $Q(\Omega_\Gamma) = 1$.
\end{lemma}

\begin{proof}
  Since
  \begin{equation*}
    \PP_\Gamma \setminus \Omega_\Gamma = \bigcup_{\mu \in \PP(\Gamma)} \{ \mu \} \times (\Gamma \setminus \spt(\mu)) = \{ (\mu,\iii) \in \PP_\Gamma : \iii \notin \spt(\mu) \}
  \end{equation*}
  Lemma \ref{thm:adapted_def} gives
  \begin{equation*}
    \int_{\PP_\Gamma} \chi_{\PP_\Gamma \setminus \Omega_\Gamma}(\mu,\iii) \dd Q(\mu,\iii) = \int_{\PP(\Gamma)} \int_\Gamma \chi_{\Gamma \setminus \spt(\mu)}(\iii) \dd\mu(\iii)\dd\overline{Q}(\mu) = 0
  \end{equation*}
  as claimed.
\end{proof}

We have learnt the following observation from \cite{Shmerkin2011}. We sketch the proof for the convenience of the reader.

\begin{lemma} \label{thm:adapted2}
  A distribution $Q$ on $\PP_\Gamma$ is adapted if and only if there exists a distribution $\overline{Q}$ on $\PP(\Gamma)$ such that
  \begin{equation} \label{eq:adapted2}
    \int_{\PP_\Gamma} f(\mu)g(\iii) \dd Q(\mu,\iii) = \int_{\PP(\Gamma)} f(\mu) \int_\Gamma g(\iii) \dd\mu(\iii) \dd\overline{Q}(\mu)
  \end{equation}
  for all $f \in C(\PP(\Gamma))$ and $g \in C(\Gamma)$.
\end{lemma}

\begin{proof}
  Since adaptedness clearly implies \eqref{eq:adapted2} let us assume that there exists a distribution $\overline{Q}$ on $\PP(\Gamma)$ such that \eqref{eq:adapted2} holds for all functions in $C(\PP(\Gamma))$ and $C(\Gamma)$. Fix $f \in C(\PP_\Gamma)$ and for each $n \in \N$ define a function $f_n \colon \PP_\Gamma \to \R$ by setting
  \begin{equation*}
    f_n(\mu,\iii) = \sum_{\jjj \in \Gamma_n} \min\{ f(\mu,\hhh) : \hhh \in [\jjj] \} \cdot \chi_{[\jjj]}(\iii)
  \end{equation*}
  for all $(\mu,\iii) \in \PP_\Gamma$. Observe that $(f_n)_{n \in \N}$ is an increasing sequence of continuous functions such that $f_n \to f$ as $n \to \infty$. Moreover, each $f_n$ is defined to be a sum of products, where each product is between functions in $C(\PP(\Gamma))$ and $C(\Gamma)$. Therefore, the proof follows by applying \eqref{eq:adapted2} and the monotone convergence theorem.
\end{proof}

\begin{lemma} \label{thm:convex_compact}
  The family of adapted distributions on $\PP_\Gamma$ is convex and compact with respect to the weak topology.
\end{lemma}

\begin{proof}
  It is straightforward to see that the family of adapted distributions on $\PP_\Gamma$ is convex. To show the compactness, rely on Lemma \ref{thm:adapted2} and observe that the function $\mu \mapsto \int_\Gamma g(\iii) \dd\mu(\iii)$ defined on $\PP(\Gamma)$ is continuous.
\end{proof}

\subsection{CP-distributions}
Let $\Gamma \subset \Sigma$ be a compact set satisfying $\sigma(\Gamma) \subset \Gamma$. If $\mu \in \PP(\Gamma)$, then for each $\iii \in \Sigma_*$ we define a measure $\mu_\iii \in \PP(\Gamma)$ by setting
\begin{equation*}
  \mu_\iii([\jjj]) =
  \begin{cases}
    \mu([\iii\jjj])/\mu([\iii]), &\text{if } \iii\jjj \in \Gamma_* \text{ and } \mu([\iii])>0, \\
    0, &\text{otherwise},
  \end{cases}
\end{equation*}
for all $\jjj \in \Sigma_*$. We define a mapping $M \colon \PP_\Gamma \to \PP_\Gamma$ by setting $M(\mu,\iii) = (\mu_{\iii|_1},\sigma(\iii))$ for all $(\mu,\iii) \in \PP_\Gamma$. Observe that if $\iii\jjj\hhh \in \Gamma_*$ and $\mu([\iii\jjj])>0$, then
\begin{equation*}
  (\mu_\iii)_\jjj([\hhh]) = \frac{\mu_\iii([\jjj\hhh])}{\mu_\iii([\jjj])} = \frac{\mu([\iii\jjj\hhh])}{\mu([\iii\jjj])} = \mu_{\iii\jjj}([\hhh]).
\end{equation*}
Thus $M^k(\mu,\iii) = (\mu_{\iii|_k},\sigma^k(\iii))$ for all $(\mu,\iii) \in \PP_\Gamma$ and $k \in \N$. Furthermore, if $\mu([\iii|_n]) > 0$ for all $n \in \N$, then also $\mu_{\iii|_1}([\sigma(\iii)|_n]) > 0$ for all $n \in \N$. Hence $M(\Omega_\Gamma) \subset \Omega_\Gamma$ where $\Omega_\Gamma$ is as in \eqref{eq:omegagamma}. It follows that $M$ restricted to $\Omega_\Gamma$ is continuous.

The following lemma can be gleaned from \cite{Furstenberg2008}. Nevertheless, we present a full proof for it since we have not been able to find it anywhere in the literature.

\begin{lemma} \label{thm:invariant_adapted}
  If $Q$ is an adapted distribution on $\PP_\Gamma$, then $MQ$ is an adapted distribution on $\PP_\Gamma$.
\end{lemma}

\begin{proof}
  Fix $f \in C(\PP(\Gamma))$ and $g \in C(\Gamma)$. For each $n \in \N$ define a function $g_n \colon \Gamma \to \R$ by setting
  \begin{equation*}
    g_n(\iii) = \sum_{\jjj \in \Gamma_n} \min\{ g(\hhh) : \hhh \in [\jjj] \} \cdot \chi_{[\jjj]}
  \end{equation*}
  for all $\iii \in \Gamma$. Observe that $(g_n)_{n \in \N}$ is an increasing sequence of continuous functions such that $g_n \to g$ as $n \to \infty$. Thus, if we are able to show that there exists a distribution $\hat Q$ on $\PP(\Gamma)$ such that
  \begin{equation*}
    \int_{\PP_\Gamma} f(\mu) g_n(\iii) \dd MQ(\mu,\iii) = \int_{\PP(\Gamma)} f(\mu) \int_\Gamma g_n(\iii) \dd\mu(\iii) \dd\hat Q(\mu)
  \end{equation*}
  for all $n \in \N$, then the claim follows from the monotone convergence theorem and Lemma \ref{thm:adapted2}.
  
  Fix $n \in \N$ and define a linear operator $T \colon C(\PP(\Gamma)) \to C(\PP(\Gamma))$ by setting
  \begin{equation*}
    Th(\mu) = \sum_{i \in \Gamma_1} \mu([i]) h(\mu_i)
  \end{equation*}
  for all $h \in C(\PP(\Gamma))$. By the Riesz representation theorem, 
  there exists an adjoint operator $T^* \colon \PP(\PP(\Gamma)) \to \PP(\PP(\Gamma))$ such that
  \begin{equation} \label{eq:adjoint}
    \int_{\PP(\Gamma)} Th(\mu) \dd\overline{Q}(\mu) = \int_{\PP(\Gamma)} h(\mu) \dd T^*\overline{Q}(\mu)
  \end{equation}
  for all $h \in C(\PP(\Gamma))$ and $\overline{Q} \in \PP(\PP(\Gamma))$.
  Let $h \in C(\PP(\Gamma))$ be so that
  \begin{equation*}
    h(\mu) = f(\mu) \int_\Gamma g_n(\iii) \dd\mu(\iii)
  \end{equation*}
  for all $\mu \in \PP(\Gamma)$. Recalling that $Q$ is adapted, we choose $\overline{Q} \in \PP(\PP(\Gamma))$ so that it satisfies \eqref{eq:adapted}. Since each $\mu \in \PP(\Gamma)$ is a measure on $\Sigma$ with $\spt(\mu) \subset \Gamma$ we get
  \begin{equation} \label{eq:lasku}
  \begin{split}
    Th(\mu) &= \sum_{i \in \Gamma_1} \mu([i]) h(\mu_i)
    = \sum_{i \in \Gamma_1} \mu([i]) f(\mu_i) \sum_{\jjj \in \Gamma_n} \min\{ g(\hhh) : \hhh \in [\jjj] \} \cdot \mu_i([\jjj]) \\
    &= \sum_{i \in \Gamma_1} f(\mu_i) \sum_{\jjj \in \Gamma_n} \min\{ g(\hhh) : \hhh \in [\jjj] \} \cdot \mu([i\jjj])
    = \sum_{i \in \Gamma_1} f(\mu_i) \int_{[i]} g_n(\sigma(\iii)) \dd\mu(\iii) \\
    &= \int_\Gamma f(\mu_{\iii|_1}) g(\sigma(\iii)) \dd\mu(\iii)
  \end{split}
  \end{equation}
  for all $\mu \in \PP(\Gamma)$.
  Observe that since $\mu_n \to \mu$ weakly if and only if $\mu_n([\iii]) \to \mu([\iii])$ for all $\iii \in \Gamma_*$ the function $\mu \mapsto f(\mu_i)$ is continuous whenever $\mu([i])>0$. Thus the bounded function $(\mu,\iii) \mapsto f(\mu_{\iii|_1}) g(\sigma(\iii))$ defined on $\PP_\Gamma$ is measurable since its restriction to $\Omega_\Gamma$ is continuous and $\Omega_\Gamma$ has full measure by Lemma \ref{thm:adapted_omega}. Now, by Lemma \ref{thm:adapted_def}, \eqref{eq:lasku}, and \eqref{eq:adjoint}, we get
  \begin{align*}
    \int_{\PP_\Gamma} f(\mu) g_n(\iii) \dd MQ(\mu,\iii) &= \int_{\PP_\Gamma} f(\mu_{\iii|_1}) g(\sigma(\iii)) \dd Q(\mu,\iii) \\
    &= \int_{\PP(\Gamma)} \int_\Gamma f(\mu_{\iii|_1}) g(\sigma(\iii)) \dd\mu(\iii) \dd\overline{Q}(\mu) \\
    &= \int_{\PP(\Gamma)} Th(\mu) \dd\overline{Q}(\mu)
    = \int_{\PP(\Gamma)} h(\mu) \dd T^*\overline{Q}(\mu) \\
    &= \int_{\PP(\Gamma)} f(\mu) \int_\Gamma g_n(\iii) \dd\mu(\iii) \dd T^*\overline{Q}(\mu).
  \end{align*}
  This is what we wanted to show.
\end{proof}

If $Q$ is a distribution on $\PP_\Gamma$, then we say that $A \subset \PP_\Gamma$ is an \emph{invariant set} if $Q(M^{-1}(A) \triangle A) = 0$. It follows from Lemmas \ref{thm:invariant_adapted} and \ref{thm:adapted_omega} that $\Omega_\Gamma$ is an invariant set whenever $Q$ is adapted. Furthermore, we say that $Q$ is \emph{invariant} if $MQ=Q$ and \emph{ergodic} if $Q(A) \in \{ 0,1 \}$ for all invariant sets $A \subset \PP_\Gamma$. An invariant and adapted distribution $Q$ on $\PP_\Gamma$ is called a \emph{CP-distribution}.

If $\mu \in \PP(\Gamma)$ and $\iii \in \Gamma$, then we define
\begin{equation*}
  \la \mu,\iii \ra_n = \tfrac{1}{n} \sum_{k=0}^{n-1} \delta_{M^k(\mu,\iii)}
\end{equation*}
for all $n \in \N$. We say that a measure $\mu$ on $\Gamma$ \emph{generates} a distribution $Q$ on $\PP_\Gamma$ if $\lim_{n \to \infty} \la \mu,\iii \ra_n = Q$ for $\mu$-almost all $\iii \in \Gamma$. It is shown in \cite[Proposition 5.4]{Hochman2010} that generated distributions are adapted. This observation is a crucial ingredient in the following lemma.

\begin{lemma}
  If $Q$ is a CP-distribution on $\PP_\Gamma$, then its ergodic components are $CP$-distributions.
\end{lemma}

\begin{proof}
  By the ergodic decomposition, each ergodic component $Q_\omega$ is invariant. Thus it suffices to show that an ergodic and invariant distribution is adapted. If $Q$ is such a distribution on $\PP_\Gamma$ and $f$ is a continuous function defined on $\PP_\Gamma$, then, by the Birkhoff ergodic theorem, we have
  \begin{equation*}
    \lim_{n \to \infty} \int_{\PP_\Gamma} f(\nu,\jjj) \dd \la \mu,\iii \ra_n(\nu,\jjj) = \lim_{n \to \infty} \tfrac{1}{n} \sum_{k=0}^{n-1} f(M^k(\mu,\iii)) = \int_{\PP_\Gamma} f(\nu,\jjj) \dd Q(\nu,\jjj)
  \end{equation*}
  for $Q$-almost all $(\mu,\iii)$. Since this holds simultaneously for any countable collection of continuous functions and $C(\PP_\Gamma)$ is separable in the uniform norm, it holds simultaneously for all continuous $f$. Therefore $\overline{Q}$-almost every $\mu$ generates $Q$ and, by \cite[Proposition 5.4]{Hochman2010}, $Q$ is adapted.
\end{proof}

\subsection{CP-distributions and entropy}
We define an \emph{information function} $I \colon \PP_\Gamma \to \R$ by setting $I(\mu,\iii) = -\log\mu([\iii|_1])$ for all $(\mu,\iii) \in \Omega_\Gamma$ and $I(\mu,\iii) = 0$ for all $(\mu,\iii) \in \PP_\Gamma \setminus \Omega_\Gamma$. It is evident that $I$ is continuous on $\Omega_\Gamma$.
Furthermore, we define an \emph{entropy} of a measure $\mu \in \PP(\Gamma)$ by setting
\begin{equation*}
  h(\mu) = -\lim_{n \to \infty} \tfrac{1}{n} \log \mu([\iii|_n])
\end{equation*}
whenever the limit exists and is $\mu$-almost everywhere constant. By the Shannon-McMillan-Breiman Theorem, this is the case for all $\sigma$-invariant ergodic measures in $\PP(\Gamma)$. The next observation follows from the Birkhoff ergodic theorem; consult the proof of \cite[Theorem 2.1]{Furstenberg2008}.

\begin{lemma} \label{thm:ergodic-entropy}
  If $Q$ is an ergodic CP-distribution on $\PP_\Gamma$, then
  \begin{equation*}
    h(\mu) = \int_{\PP_\Gamma} I(\nu,\jjj) \dd Q(\nu,\jjj)
  \end{equation*}
  for $\overline{Q}$-almost all $\mu$.
\end{lemma}

Although the following proposition is essentially \cite[Proposition 5.2]{Furstenberg2008}, we present its proof in detail since it nicely binds together the crucial ideas. Define
\begin{equation*}
  \Omega_\Gamma(A) = \{ (\mu,\iii) \in \Omega_\Gamma : \spt(\mu) \text{ is contained in a microset of } A \}
\end{equation*}
for all compact $A \subset \Gamma$.

\begin{proposition} \label{thm:exists-invariant}
  If $A \subset \Gamma$ is compact, then there exists a CP-distribution $Q$ on $\PP_\Gamma$ such that $Q(\Omega_\Gamma(A)) = 1$ and
  \begin{equation*}
    \int_{\PP_\Gamma} I(\nu,\jjj) \dd Q(\nu,\jjj) = \lim_{n \to \infty} \tfrac{1}{n} \log N_n(A).
  \end{equation*}
\end{proposition}

\begin{proof}
  For each $n \in \N$ let $A_n'$ be a microset of $A$ such that $N_n(A) = \#\{ \iii \in \Gamma_n : A_n' \cap [\iii] \ne \emptyset \}$. Let $\mu_n \in \PP(\Gamma)$ be a measure supported on $A_n'$ such that $\mu_n([\iii]) = N_n(A)^{-1}$ for all $\iii \in \Gamma_n$ with $A_n' \cap [\iii] \ne \emptyset$. Define $P_n = \delta_{\mu_n} \times \mu_n$ for all $n \in \N$. Since
  \begin{equation} \label{eq:Pn_adapted}
    \int_{\PP_\Gamma} f(\nu,\iii) \dd P_n(\nu,\iii) = \int_{\PP(\Gamma)} \int_\Gamma f(\nu,\iii) \dd\nu(\iii) \dd\delta_{\mu_n}(\nu)
  \end{equation}
  for all $f \in C(\PP_\Gamma)$ the distribution $P_n$ is clearly adapted. Furthermore, let $Q_n = \tfrac{1}{n} \sum_{k=0}^{n-1} M^k P_n$ for all $n \in \N$ and choose $Q$ to be an accumulation point of $\{ Q_n \}_{n \in \N}$. Observe that $Q$ is invariant. It follows from Lemmas \ref{thm:invariant_adapted} and \ref{thm:convex_compact} that $Q$ is adapted. Thus $Q$ is a CP-distribution.
  
  Let us prove that $Q(\Omega_\Gamma(A)) = 1$. We will first show that $Q_n(\Omega_\Gamma(A)) = 1$ for all $n \in \N$. Then, by showing that $\Omega_\Gamma(A)$ is closed, the claim follows from the weak convergence since $Q(\Omega_\Gamma(A)) \ge \lim_{n \to \infty} Q_n(\Omega_\Gamma(A)) = 1$.

  To show that $Q_n(\Omega_\Gamma(A))=1$ for all $n \in \N$ it suffices to prove that for fixed $k \in \{ 0,\ldots,n-1 \}$ it holds $M^kP_n(\Omega_\Gamma(A))=1$. Recalling the definition of $P_n$, this is certainly the case if $\{ \mu_n \} \times A_n' \subset M^{-k}(\Omega_\Gamma(A))$. Fix $\iii \in A_n'$ and note that the measure $(\mu_n)_{\iii|_k}$ is supported on the miniset $\beta_{\iii|_k}(A_n')$. Since, by Lemma \ref{thm:micromicro}, $\beta_{\iii|_k}(A_n')$ is a microset of $A$ we have $M^k(\mu_n,\iii) = ((\mu_n)_{\iii|_k},\sigma^k(\iii)) \in \Omega_\Gamma(A)$. This is what we wanted.

  To show that $\Omega_\Gamma(A)$ is closed take a sequence $(\nu_j,\iii_j)_{j=1}^\infty$ of points in $\Omega_\Gamma(A)$ converging to some $(\nu,\iii) \in \PP_\Gamma$. For each $j \in \N$ let $A_j'$ be a microset of $A$ for which $\spt(\nu_j) \subset A_j'$. Observe that there exists a sequence $(\jjj_k^j)_{k=1}^\infty$ of finite words such that $\beta_{\jjj_k^j}(A) \to A_j'$ as $k \to \infty$. By the compactness of the Hausdorff metric $D$, the sequence $(A_j')_{j=1}^\infty$ has a converging subsequence. Since there is no danger of misunderstanding we keep denoting the subsequence by $(A_j')_{j=1}^\infty$. Let $A' \subset \Gamma$ be such that $A_j' \to A'$ as $j \to \infty$.
  
  Let $\eps>0$ and choose $j_0 \in \N$ such that $D(A_j',A') < \eps/2$ for all $j \ge j_0$. For each $j \ge j_0$ let $k \in \N$ be such that $D(\beta_{\jjj_k^j}(A),A_j') < \eps/2$. Since
  \begin{equation*}
    D(\beta_{\jjj_k^j}(A),A') \le D(\beta_{\jjj_k^j}(A),A_j') + D(A_j',A') < \eps
  \end{equation*}
  we see that also $A'$ is a microset of $A$. Assume contrarily that $\spt(\nu)$ is not contained in $A'$. Then there are $\kkk \in \Gamma_*$ and $j_0 \in \N$ such that $\nu([\kkk])>0$ and $\nu_j([\kkk])=0$ for all $j \ge j_0$. But this cannot be the case since cylinder sets are open and hence $0 = \liminf_{j \to \infty} \nu_j([\kkk]) \ge \nu([\kkk])>0$. Thus $\spt(\nu) \subset A'$ and $(\nu,\iii) \in \Omega_A$. Therefore $\Omega_\Gamma(A)$ is closed and $Q(\Omega_\Gamma(A))=1$.
  
  To show the second claim, observe that, by the definition of $Q_n$ and \eqref{eq:Pn_adapted}, we have
  \begin{align*}
    \int_{\PP_\Gamma} I(\nu,\jjj) \dd Q_n(\nu,\jjj) &= \tfrac{1}{n} \int_\Gamma \sum_{k=0}^{n-1} I(M^k(\mu_n,\iii)) \dd\mu_n(\iii)
    = -\tfrac{1}{n} \int_\Gamma \sum_{k=0}^{n-1} \log\frac{\mu_n([\iii|_{k+1}])}{\mu_n([\iii|_k])} \dd\mu_n(\iii) \\
    &= -\tfrac{1}{n} \int_\Gamma \log\mu_n([\iii|_n]) \dd\mu_n(\iii)
    = \tfrac{1}{n} \log N_n(A).
  \end{align*}
  By letting $n \to \infty$, the claim follows from \cite[Theorem 2.7]{Billingsley1999} since $I$ is continuous on $\Omega_\Gamma$ and $Q(\Omega_\Gamma)=1$.
\end{proof}

The following is an immediate corollary of Proposition \ref{thm:exists-invariant}, the ergodic decomposition, and Lemma \ref{thm:ergodic-entropy}.

\begin{corollary} \label{thm:entropy}
  If $A \subset \Gamma$ is compact, then there exists an ergodic CP-distribution $Q$ on $\PP_\Gamma$ such that for $\overline{Q}$-almost every $\mu$ the support of $\mu$ is contained in a microset of $A$ and
  \begin{equation*}
    \lim_{n \to \infty} \tfrac{1}{n} \log N_n(A) \le h(\mu).
  \end{equation*}
\end{corollary}

\section{Microsets and dimension} \label{sec:micro_and_dim}
Our goal is to understand dimensional properties of Moran constructions. By relying on the main result of the previous section, Corollary \ref{thm:entropy}, we can now relate the Assouad dimension to the Hausdorff dimension. As a particular outcome of this observation, we are able to show that Furstenberg homogeneous self-similar sets in the real line satisfy the uniform weak separation condition.

\begin{proposition} \label{thm:micro-dimensiot}
  Suppose that $X$ is a complete metric space, $\Gamma \subset \Sigma$ is a compact set satisfying $\sigma(\Gamma) \subset \Gamma$, and $\{ E_\iii : \iii \in \Gamma_* \}$ is a Moran construction. If the Moran construction satisfies the uniform finite clustering property and there are constants $C \ge 1$ and $0<\ualpha<1$ so that
  \begin{equation} \label{eq:moran-homogeneous}
    C^{-1}\ualpha^{|\iii|} \le \diam(E_\iii) \le C\ualpha^{|\iii|}
  \end{equation}
  for all $\iii \in \Gamma_*$, then
  \begin{equation*}
    \sup\{ \dima(\pi(A')) : A' \text{ is a microset of } A \} = \sup\{ \dimh(\pi(A')) : A' \text{ is a microset of } A \}
  \end{equation*}
  for all compact sets $A \subset \Gamma$.
\end{proposition}

\begin{proof}
  Recall that by Proposition \ref{thm:dimf} we have
  \begin{equation*}
    \dima(\pi(A')) \le \lim_{n \to \infty} \frac{\log N_n(A)}{\log\ualpha^{-n}}
  \end{equation*}
  for all microsets $A'$ of $A$. The assumption \eqref{eq:moran-homogeneous} together with Proposition \ref{thm:dimloc} give
  \begin{equation*}
    \liminf_{n \to \infty} \frac{\log\mu([\iii|_n])}{\log\ualpha^n} = \ldimh(\pi\mu)
  \end{equation*}
  for any Borel probability measure $\mu$. According to Corollary \ref{thm:entropy}, there exists an ergodic CP-distribution $Q$ on $\PP_\Gamma$ such that for $\overline{Q}$-almost every $\mu$ the support of $\mu$ is contained in a microset of $A$ and
  \begin{equation*}
    \lim_{n \to \infty} \tfrac{1}{n} \log N_n(A) \le -\lim_{n \to \infty} \tfrac{1}{n} \log\mu([\iii|_n]).
  \end{equation*}
  Since $\spt(\mu) \subset A'$ for some microset $A'$ of $A$ we have $\ldimh(\pi\mu) \le \dimh(A')$. The proof follows by combining all the above estimates.
\end{proof}

The following is an immediate corollary of Propositions \ref{thm:micro-dimensiot} and \ref{thm:dima_cifs}.

\begin{corollary} \label{thm:cifs-micro-dimensiot}
  Suppose that $\{ \fii_i \}_{i=1}^\kappa$ is a conformal iterated function system, $\Gamma \subset \Sigma$ is a compact set satisfying $\sigma(\Gamma) \subset \Gamma$, and $\{ \fii_\iii(W) : \iii \in \Gamma_* \}$ is a Moran construction as in Lemma \ref{thm:ifs}. If the Moran construction satisfies the uniform finite clustering property and there are constants $C \ge 1$ and $0<\ualpha<1$ so that
  \begin{equation*}
    C^{-1}\ualpha^{|\iii|} \le \diam(\fii_\iii(W)) \le C\ualpha^{|\iii|}
  \end{equation*}
  for all $\iii \in \Gamma_*$, then
  \begin{equation*}
    \dima(\pi(A)) = \sup\{ \dimh(\pi(A')) : A' \text{ is a microset of } A \}.
  \end{equation*}
  for all compact sets $A \subset \Gamma$.
\end{corollary}

Observe that, by Proposition \ref{thm:wsc}, we are in the setting of Corollary \ref{thm:cifs-micro-dimensiot} if the conformal iterated function system satisfies the uniform weak separation condition. Recall also that, by Remark \ref{rem:wsc}(2), the open set condition implies the uniform weak separation condition.

We will now define microsets in $\R^d$. To distinguish this definition from the symbolic one given in \S \ref{sec:microset}, we will call them Furstenberg microsets. Let $Q = [0,1]^d$. A set $A' \subset Q$ is a \emph{Furstenberg miniset} of $A \subset Q$ if $A' \subset (\lambda A + t) \cap Q = \{ \lambda x + t : x \in A \} \cap Q$ for some $\lambda \ge 1$ and $t \in \R^d$. A set $A' \subset Q$ is a \emph{Furstenberg microset} of a compact set $A \subset Q$ if there exists a sequence $(A_n')_{n=1}^\infty$ of Furstenberg minisets of $A$ such that $A_n' \to A'$ in the Hausdorff metric. Note that a Furstenberg miniset of a compact set is clearly a Furstenberg microset.

\begin{remark} \label{rem:microset_is_microset}
  Suppose that $\{ \fii_i \}_{i=1}^\kappa$ is a homothetic iterated function system, $\Gamma \subset \Sigma$ is a compact set satisfying $\sigma(\Gamma) \subset \Gamma$, and $\{ \fii_\iii(W) : \iii \in \Gamma_* \}$ is a Moran construction as in Lemma \ref{thm:ifs}. We assume that $W$ can be chosen to be $Q$. Then the self-homothetic set $E$ is contained in $Q$. Observe that, since each $\fii_i$ is a contracting homothety, for each $\iii \in \Sigma_*$ there are $\lambda_\iii > 1$ and $t_\iii \in \R^d$ such that $\fii_\iii^{-1}(x) = \lambda_\iii x + t_\iii$ for all $x \in \R^d$. If $A' \subset \Gamma$ is a miniset of a given set $A \subset \Gamma$, then $A' = \sigma^{|\iii|}(A \cap [\iii])$ for some $\iii \in \Sigma_*$ and thus
  \begin{equation*}
    \pi(A') \subset \fii_\iii^{-1}(\pi(A)) \cap E \subset (\lambda_\iii \pi(A) + t_\iii) \cap Q.
  \end{equation*}
  Therefore, for a homothetic iterated function system, a projected miniset is a Furstenberg miniset. Consequently, since the projection is continuous, a projected microset is a Furstenberg microset.
\end{remark}

We say that a set $A \subset Q$ is \emph{Furstenberg homogeneous} if every Furstenberg microset of $A$ is a Furstenberg miniset of $A$. The following proposition is proved in \cite[Claim 9.6]{Hochman2012}. Recall that a self-similar set $E$ satisfies the \emph{strong separation condition} if $\fii_i(E) \cap \fii_j(E) = \emptyset$ whenever $i \ne j$, where $\fii_i$'s are the similarities of the associated iterated function system. It is well known that the strong separation condition implies the open set condition.

\begin{proposition}
  A self-homothetic set $E \subset Q$ satisfying the strong separation condition is Furstenberg homogeneous.
\end{proposition}

An interesting and completely open question is to find a characterization for Furstenberg homogeneous sets. By relying on Corollary \ref{thm:cifs-micro-dimensiot} and \cite[Theorem 3.1]{FraserHendersonOlsonRobinson2015}, we are now able to study this question.

\begin{theorem}
  A Furstenberg homogeneous self-similar set $E \subset [0,1]$ with $\dimh(E) < 1$ satisfies the uniform weak separation condition.
\end{theorem}

\begin{proof}
  Let $\{ \fii_1,\fii_2 \}$ be a homothetic iterated function system so that $\fii_1 \colon \R \to \R$, $\fii_1(x) = \tfrac12 x$, and $\fii_2 \colon \R \to \R$, $\fii_2(x) = \tfrac12 x + \tfrac12$. Then the Moran construction $\{ \fii_\iii([0,1]) : \iii \in \Sigma_* \}$ satisfies the assumptions of Corollary \ref{thm:cifs-micro-dimensiot} and the associated self-homothetic set is $[0,1]$.
  
  For each compact $E \subset [0,1]$ there is a compact set $A \subset \Sigma$ such that $\pi(A)=E$. By Corollary \ref{thm:cifs-micro-dimensiot}, we thus have
  \begin{equation*}
    \dima(E) = \sup\{ \dimh(\pi(A')) : A' \text{ is a microset of } A \}.
  \end{equation*}
  Let $A' \subset \Sigma$ be a microset of $A$. By Remark \ref{rem:microset_is_microset}, the projection $\pi(A')$ is a Furstenberg microset of $E$. Therefore, if $E$ is Furstenberg homogeneous, then $\pi(A')$ is a Furstenberg miniset. By the definition of Furstenberg homogeneity and the standard properties of the Hausdorff dimension, we thus have $\dimh(\pi(A')) \le \dimh(E)$. Therefore, the assumption $\dimh(E)<1$, together with the Furstenberg homogeneity, implies $\dima(E)<1$.
  
  By \cite[Theorem 3.1]{FraserHendersonOlsonRobinson2015}, we know that if $E \subset \R$ is a self-similar set with $\dima(E)<1$, then it satisfies the uniform weak separation condition; recall \cite[Theorem 1]{Zerner1996}. This finishes the proof.
\end{proof}

In light of the above result, it is now interesting to ask whether the weak separation condition characterizes Furstenberg homogeneous self-homothetic sets in the line. The following example shows that this is not the case. Observe that it is certainly possible for a self-homothetic set which satisfies the weak separation condition but does not satisfy the open set condition to be Furstenberg homogeneous. For example, duplicating one mapping in any homothetic iterated function system satisfying the strong separation condition serves as a simple example.

\begin{example} \label{ex:furstenberg-OSC}
  We exhibit a self-homothetic set which satisfies the open set condition but fails to be Furstenberg homogeneous. For each $i \in \{ 1,2,3 \}$ define $\fii_i \colon \R \to \R$ by setting
  \begin{equation*}
    \fii_1(x) = \tfrac12 x, \quad \fii_2(x) = \tfrac15 x + \tfrac12, \quad \fii_3(x) = \tfrac17 x + \tfrac67,
  \end{equation*}
  for all $x \in \R$. It is evident that the iterated function system $\{ \fii_i \}_{i=1}^3$ satisfies the open set condition. Let $E$ be the associated self-similar set. Observe that $E \subset [0,1]$ with $\diam(E)=1$. Furthermore, if $\iii \bot \jjj$ and $\fii_\iii(E) \cap \fii_\jjj(E) \ne \emptyset$, then the intersection consists of a single point $\fii_{\iii \wedge \jjj}(\tfrac12)$. Since $a^n \ne b^m$ for all $n,m \in \N$ and $a,b \in \{ 2,5,7 \}$ with $a \ne b$ we see that
  \begin{equation} \label{eq:different_diam}
  \begin{split}
    \diam(\fii_\iii(E)) &= \diam(\fii_{\iii \wedge \jjj}(E)) \cdot \tfrac12 \bigl( \tfrac17 \bigr)^{|\iii|-|\iii \wedge \jjj|-1} \\ &\ne \diam(\fii_{\iii \wedge \jjj}(E)) \cdot \tfrac15 \bigl( \tfrac12 \bigr)^{|\jjj|-|\iii \wedge \jjj|-1} = \diam(\fii_\jjj(E))
  \end{split}
  \end{equation}
  whenever $\iii \bot \jjj$ and $\fii_\iii(E) \cap \fii_\jjj(E) \ne \emptyset$.
  
  Since $\log_2 7 \in \R \setminus \Q$ the set $\{ m\log_2 7 \mod 1 : m \in \N \}$ is dense in $[0,1]$. Thus there exist increasing sequences $(n_j)_{j \in \N}$ and $(m_j)_{j \in \N}$ of integers so that
  \begin{equation*}
    n_j \le 1 + m_j \log_2 7 - \log_2 5 < n_j + j^{-1}.
  \end{equation*}
  Define $\overline{i} = iii\cdots$ for all $i \in \{ 1,2,3 \}$. Observe that the interval $[\tfrac12 - \tfrac12 (\tfrac17)^{m_j}, \tfrac12 + \tfrac12 (\tfrac17)^{m_j}]$ intersects the sets $\fii_{1(\overline{3}|_{m_j})}(E)$ and $\fii_{2(\overline{1}|_{n_j})}(E)$. The center of the interval $\tfrac12$ is the largest point of the set $\fii_{1(\overline{3}|_{m_j})}(E)$ and the smallest point of the set $\fii_{2(\overline{1}|_{n_j})}(E)$. Since
  \begin{equation*}
    \diam(\fii_{1(\overline{3}|_{m_j})}(E)) = \tfrac12 \bigl( \tfrac17 \bigr)^{m_j}
  \end{equation*}
  and
  \begin{equation*}
    \diam(\fii_{2(\overline{1}|_{n_j})}(E)) = \tfrac15 \bigl( \tfrac12 \bigr)^{n_j} \ge \tfrac12 \bigl( \tfrac12 \bigr)^{1+m_j\log_27-\log_25} = \tfrac12 \bigl( \tfrac17 \bigr)^{m_j}
  \end{equation*}
  there is no $\iii \in \Sigma_*$ with $\iii \bot 1(\overline{3}|_{m_j})$ and $\iii \bot 2(\overline{1}|_{n_j})$ so that $\fii_\iii(E)$ intersects $[\tfrac12 - \tfrac12 (\tfrac17)^{m_j}, \tfrac12 + \tfrac12 (\tfrac17)^{m_j}]$.
  
  For each $0 \le u<v\le 1$ we define a mapping $M_{u,v} \colon \R \to \R$ by setting
  \begin{equation*}
    M_{u,v}(x) = \frac{x-u}{v-u}
  \end{equation*}
  for all $x \in \R$. Thus $M_{u,v}$ is a homothety taking $[u,v]$ to $[0,1]$. Let $u_j = \tfrac12 - \tfrac12 (\tfrac17)^{m_j}$ and $v_j = \tfrac12 + \tfrac12 (\tfrac17)^{m_j}$ for all $j \in \N$. Then $M_{u_j,v_j}$ magnifies by a constant $7^{m_j}$ so that $M_{u_j,v_j}(\tfrac12) = \tfrac12$. Recall that we have
  \begin{equation*}
    7^{m_j} \diam(\fii_{1(\overline{3}|_{m_j})}(E)) = \tfrac12
  \end{equation*}
  and since
  \begin{equation*}
    \tfrac12 \le 7^{m_j}\diam(\fii_{2(\overline{1}|_{n_j})}(E)) \le 7^{m_j}\tfrac15\bigl( \tfrac12 \bigr)^{1+m_j\log_27-\log_25-j^{-1}} = \tfrac12 \sqrt[j]{2}
  \end{equation*}
  we have
  \begin{equation*}
    7^{m_j}\diam(\fii_{2(\overline{1}|_{n_j})}(E)) \to \tfrac12
  \end{equation*}
  as $j \to \infty$. Since $\fii_{1(\overline{3}|_{m_j})}(E) \cap \fii_{2(\overline{1}|_{n_j})}(E) = \{ \tfrac12 \}$ it therefore holds that
  \begin{equation*}
    M_{u_j,v_j}(E \cap [u_j,v_j]) \to \tfrac12 E \cup (\tfrac12 E + \tfrac12)
  \end{equation*}
  in the Hausdorff metric as $j \to \infty$. Thus the set $\tfrac12 E \cup (\tfrac12 E + \tfrac12)$ is a microset. By showing that it is not a miniset, it follows that $E$ is not Furstenberg homogeneous. Since, by \eqref{eq:different_diam}, any two sets $\fii_\iii(E)$ next to each other have different lengths, the set $\tfrac12 E \cup (\tfrac12 E + \tfrac12)$ cannot be a miniset whose center is at the intersection of the magnification of these sets. For this reason the claim is intuitively easy to believe. Nevertheless, we will provide the reader with a rigorous argument.
  
  Suppose to the contrary that $\tfrac12 E \cup (\tfrac12 E + \tfrac12) \subset M_{u,v}(E \cap [u,v])$ where $u = \pi(\iii)$ and $v = \pi(\jjj)$. By shifting, we may assume that $\iii \wedge \jjj = \varnothing$. Let
  \begin{equation*}
    H^0 = [0,1] \setminus \bigcup_{i=1}^3 \fii_i([0,1]) = (\tfrac{7}{10},\tfrac67)
  \end{equation*}
  and $\eta = \diam(H^0) = \tfrac{11}{70}$. If there exists an interval $H \subset [u,v] \setminus E$ with $\diam(H) > \tfrac12 \eta(v-u)$, then
  \begin{equation*}
    \tfrac12 E \cup (\tfrac12 E + \tfrac12) \setminus M_{u,v}(E \cap [u,v]) \ne \emptyset
  \end{equation*}
  and we have a contradiction. Our plan is to show the existence of such an interval $H$.
  
  Since $u<v$ and $\iii \wedge \jjj = \varnothing$, we see that $\iii|_1 \ne 3$ and $\jjj|_1 \ne 1$. If $\jjj|_1 = 3$, then $H^0 \subset [u,v] \setminus E$. Since $\diam(H^0) > \tfrac12 \eta (v-u)$ we may choose $H$ to be $H^0$ and we are done. Therefore, we may assume that $\jjj|_1 = 2$ and thus $\iii|_1 = 1$. Since $u<v$ we have $u \ne \tfrac12$ or $v \ne \tfrac12$. In other words, we have $\iii \ne 1\overline{3}$ or $\jjj \ne 2\overline{1}$. Interpreting $\inf(\emptyset) = \infty$ we thus have
  \begin{equation*}
    n=\inf\{ k \ge 2 : \sigma^{k-1}(\iii)|_1 \ne 3 \} < \infty \quad \text{or} \quad m=\inf\{ k \ge 2 : \sigma^{k-1}(\jjj)|_1 \ne 1 \} < \infty.
  \end{equation*}
  If $k=\infty$, then we interpret that $\diam(\fii_{\kkk|_k}(E))=0$ for all $\kkk \in \Sigma$. Recalling \eqref{eq:different_diam}, there are now two cases: either $\diam(\fii_{\iii|_{n-1}}(E)) > \diam(\fii_{\jjj|_{m-1}}(E)) \ge 0$ or $\diam(\fii_{\jjj|_{m-1}}(E)) > \diam(\fii_{\iii|_{n-1}}(E)) \ge 0$. In the first case, since $n<\infty$ and $\sigma^{n-1}(\iii)|_1 \in \{ 1,2 \}$, we have $\fii_{\iii|_{n-1}}(H^0) \subset [u,v] \setminus E$ and $\diam(\fii_{\iii|_{n-1}}(H^0)) = \eta \diam(\fii_{\iii|_{n-1}}(E)) > \tfrac12 \eta(v-u)$. Thus, choosing $H$ to be $\fii_{\iii|_{n-1}}(H^0)$, we are done.
  
  Now we are left with the case in which $m<\infty$ and $\diam(\fii_{\jjj|_{m-1}}(E)) > \diam(\fii_{\iii|_{n-1}}(E)) \ge 0$. If $\sigma^{m-1}(\jjj)|_1 = 3$, then, choosing $H$ to be $\fii_{\jjj|_{m-1}}(H^0)$, we are done since $\diam(\fii_{\jjj|_{m-1}}(H^0)) = \eta\diam(\fii_{\jjj|_{m-1}}(E)) > \tfrac12\eta(\diam(\fii_{\iii|_{n-1}}(E))+\diam(\fii_{\jjj|_{m-1}}(E))) \ge \tfrac12\eta(v-u)$. Let us thus assume that $\sigma^{m-1}(\jjj)|_1 = 2$. With respect to the diameter of $\fii_{\iii|_{n-1}}(E)$, we get three different subcases. For simplicity, we first rescale and translate the system so that $\diam(\fii_{\jjj|_{m-1}}(E)) = 1$ and $\pi((\jjj|_{m-1})\overline{1})=0$. It follows that $u \le 0$ and $\tfrac12 \le v \le \tfrac12 + \tfrac15$.
  
  If $\diam(\fii_{\iii|_{n-1}}(E)) < \tfrac12 - \tfrac15$, then $v-u < 1$ and so $\diam(\fii_{(\jjj|_{m-1})1}(H^0)) = \eta\diam(\fii_{(\jjj|_{m-1})1}(E)) = \tfrac12 \eta > \tfrac12 \eta (v-u)$. Since $\fii_{(\jjj|_{m-1})1}(H^0) \subset [u,v] \setminus E$ we may choose $H$ to be $\fii_{(\jjj|_{m-1})1}(H^0)$ and we are done. If $\diam(\fii_{\iii|_{n-1}}(E)) > \tfrac12 + \tfrac15$, then $n < \infty$ and $\sigma^{n-1}(\iii)|_1 \in \{ 1,2 \}$. Therefore $\fii_{\iii|_{n-1}}(H^0) \subset [u,v] \setminus E$ and $\diam(\fii_{\iii|_{n-1}}(H^0)) = \eta \diam(\fii_{\iii|_{n-1}}(E)) > \tfrac12 \eta (\tfrac12+\tfrac15 + \diam(\fii_{\iii|_{n-1}}(E))) > \tfrac12 \eta (v-u)$. Choosing $H$ to be $\fii_{\iii|_{n-1}}(H^0)$ we are done.
  
  We may now assume that $\tfrac12 - \tfrac15 \le \diam(\fii_{\iii|_{n-1}}(E)) \le \tfrac12 + \tfrac15$. This subcase divides further into four subcases. If $|u|>|v|$, then $\fii_{\iii|_{n-1}}(H^0) \subset [u,v] \setminus E$ and $\diam(\fii_{\iii|_{n-1}}(H^0)) = \eta \diam(\fii_{\iii|_{n-1}}(E)) \ge \eta|u| > \tfrac12\eta(|u|+|v|) = \tfrac12 \eta(v-u)$. Therefore, choosing $H$ to be $\fii_{\iii|_{n-1}}(H^0)$ we are done. Let us then assume that $|u| \le |v|$ and $v > \tfrac12$. Defining $H = \fii_{(\jjj|_{m-1})1}(H^0) \subset [u,v] \setminus E$ we see that $\diam(H) = \eta \diam(\fii_{(\jjj|_{m-1})1}(E)) = \tfrac12 \eta = \tfrac{5}{14}\eta(1+\tfrac25) \ge \tfrac{5}{14}\eta(v-u)>\tfrac14\eta(v-u)$.
  Observe that $[0,1] \setminus (\tfrac12 E \cup (\tfrac12 E + \tfrac12))$ contains two intervals of length $\tfrac12\eta$ and the lengths of all the other intervals are at most $\tfrac14\eta$.
  Thus, regarding the size, it is possible that $M_{u,v}(H \cap [u,v])$ is contained in the complement of $\tfrac12 E \cup (\tfrac12 E + \tfrac12)$. This can only happen when
  \begin{equation*}
    M_{u,v}(w) < \tfrac12 \quad \text{or} \quad \tfrac{17}{20} = \tfrac12 + \tfrac12(\tfrac12+\tfrac15) \le M_{u,v}(w),
  \end{equation*}
  where $w = \inf(H)$. We will show that $M_{u,v}(H \cap [u,v])$ is positioned in such a way that this is not possible. Since $v-w = \tfrac12(\eta+\tfrac17) + v-\tfrac12 = \tfrac{3}{20} + v-\tfrac12$, $v-u \le 2v$, and $v>\tfrac12$ we get
  \begin{equation*}
    M_{u,v}(w) = 1 - \frac{v-w}{v-u} \le 1 - \frac{\tfrac{3}{20}+v-\tfrac12}{2v} = 1 - \tfrac{3}{20} \frac{1+\tfrac{10}{3}(2v-1)}{1+2v-1} < 1-\tfrac{3}{20} = \tfrac{17}{20}.
  \end{equation*}
  On the other hand, since $v-w = \tfrac{3}{20}+ v-\tfrac12 \le \tfrac{3}{20}+\tfrac15 = \tfrac{7}{20}$ and $w-u \ge w \ge \tfrac12(\tfrac12+\tfrac15) = \tfrac{7}{20}$ we have
  \begin{equation*}
    M_{u,v}(w)^{-1} = 1 + \frac{v-w}{w-u} \le 2
  \end{equation*}
  and hence
  \begin{equation*}
    M_{u,v}(w) \ge \tfrac{1}{2}.
  \end{equation*}
  Therefore, $\tfrac12 E \cup (\tfrac12 E + \tfrac12) \setminus M_{u,v}(E \cap [u,v]) \ne \emptyset$ also in this case.
  
  It remains to consider the case in which $|u| \le |v|$ and $v=\tfrac12$. If $|u|<|v|$, then $v-u<1$ and $\diam(\fii_{(\jjj|_{m-1})1}(H^0)) = \tfrac12\eta\diam(\fii_{\jjj|_{m-1}}(E)) = \tfrac12\eta > \tfrac12\eta(v-u)$. We are done by choosing $H$ to be $\fii_{(\jjj|_{m-1})1}(H^0)$. Finally, if $|u|=|v|=\tfrac12$, then, by \eqref{eq:different_diam}, we see that $\diam(\fii_{\iii|_{n-1}}(E)) > \diam(\fii_{(\jjj|_{m-1})1}(E)) = \tfrac12$. Since now, $\diam(\fii_{\iii|_{n-1}}(H^0)) = \eta\diam(\fii_{\iii|_{n-1}}(E) > \tfrac12\eta = \tfrac12\eta(v-u)$, we are done by choosing $H$ to be $\fii_{\iii|_{n-1}}(H^0)$.
\end{example}

\begin{ack}
  We thank M.\ Hochman for discussions related to the topics of this article.
\end{ack}

\bibliographystyle{abbrv}
\bibliography{Bibliography.bib}

\begin{thebibliography}{10}

\bibitem{BaloghRohner2007}
Z.~M. Balogh and H.~Rohner.
\newblock Self-similar sets in doubling spaces.
\newblock {\em Illinois J. Math.}, 51(4):1275--1297, 2007.

\bibitem{BaloghTysonWarhurst2009}
Z.~M. Balogh, J.~T. Tyson, and B.~Warhurst.
\newblock Sub-{R}iemannian vs. {E}uclidean dimension comparison and fractal
  geometry on {C}arnot groups.
\newblock {\em Adv. Math.}, 220(2):560--619, 2009.

\bibitem{Billingsley1999}
P.~Billingsley.
\newblock {\em Convergence of probability measures}.
\newblock Wiley Series in Probability and Statistics: Probability and
  Statistics. John Wiley \& Sons, Inc., New York, second edition, 1999.
\newblock A Wiley-Interscience Publication.

\bibitem{DengNgai2011}
Q.-R. Deng and S.-M. Ngai.
\newblock Conformal iterated function systems with overlaps.
\newblock {\em Dyn. Syst.}, 26(1):103--123, 2011.

\bibitem{Falconer1995}
K.~J. Falconer.
\newblock Sub-self-similar sets.
\newblock {\em Trans. Amer. Math. Soc.}, 347(8):3121--3129, 1995.

\bibitem{Falconer1997}
K.~J. Falconer.
\newblock {\em Techniques in Fractal Geometry}.
\newblock John Wiley {\&} Sons Ltd., England, 1997.

\bibitem{FengWenWu1997}
D.~Feng, Z.~Wen, and J.~Wu.
\newblock Some dimensional results for homogeneous {M}oran sets.
\newblock {\em Sci. China Ser. A}, 40(5):475--482, 1997.

\bibitem{Fraser2014}
J.~M. Fraser.
\newblock Assouad type dimensions and homogeneity of fractals.
\newblock {\em Trans. Amer. Math. Soc.}, 366(12):6687--6733, 2014.

\bibitem{FraserHendersonOlsonRobinson2015}
J.~M. Fraser, A.~M. Henderson, E.~J. Olson, and J.~C. Robinson.
\newblock On the {A}ssouad dimension of self-similar sets with overlaps.
\newblock {\em Adv. Math.}, 273:188--214, 2015.

\bibitem{FraserPollicott2015}
J.~M. Fraser and M.~Pollicott.
\newblock Micromeasure distributions and applications for conformally generated
  fractals.
\newblock preprint, arXiv:1502.05609, 2015.

\bibitem{Furstenberg2008}
H.~Furstenberg.
\newblock Ergodic fractal measures and dimension conservation.
\newblock {\em Ergodic Theory Dynam. Systems}, 28(2):405--422, 2008.

\bibitem{Hochman2010}
M.~Hochman.
\newblock Dynamics on fractals and fractal distributions.
\newblock preprint, arXiv:1008.3731, 2010.

\bibitem{Hochman2012}
M.~Hochman.
\newblock Lectures on fractal geometry and dynamics.
\newblock unpublished, 2012.

\bibitem{HollandZhang2013}
M.~Holland and Y.~Zhang.
\newblock Dimension results for inhomogeneous {M}oran set constructions.
\newblock {\em Dyn. Syst.}, 28(2):222--250, 2013.

\bibitem{HuaRaoWenWu2000}
S.~Hua, H.~Rao, Z.~Wen, and J.~Wu.
\newblock On the structures and dimensions of {M}oran sets.
\newblock {\em Sci. China Ser. A}, 43(8):836--852, 2000.

\bibitem{Hutchinson1981}
J.~E. Hutchinson.
\newblock Fractals and self-similarity.
\newblock {\em Indiana Univ. Math. J.}, 30(5):713--747, 1981.

\bibitem{KaenmakiRajalaSuomala2012b}
A.~K{\"a}enm{\"a}ki, T.~Rajala, and V.~Suomala.
\newblock Local multifractal analysis in metric spaces.
\newblock {\em Nonlinearity}, 26(8):2157--2173, 2013.

\bibitem{KaenmakiSahlstenShmerkin2015b}
A.~K{\"a}enm{\"a}ki, T.~Sahlsten, and P.~Shmerkin.
\newblock Dynamics of the scenery flow and geometry of measures.
\newblock {\em Proc. Lond. Math. Soc. (3)}, 110(5):1248--1280, 2015.

\bibitem{KaenmakiSahlstenShmerkin2015}
A.~K{\"a}enm{\"a}ki, T.~Sahlsten, and P.~Shmerkin.
\newblock Structure of distributions generated by the scenery flow.
\newblock {\em J. Lond. Math. Soc.}, 91:464--494, 2015.

\bibitem{KaenmakiVilppolainen2008}
A.~K{\"a}enm{\"a}ki and M.~Vilppolainen.
\newblock Separation conditions on controlled {M}oran constructions.
\newblock {\em Fund. Math.}, 200(1):69--100, 2008.

\bibitem{KaenmakiVilppolainen2010}
A.~K{\"a}enm{\"a}ki and M.~Vilppolainen.
\newblock Dimension and measures on sub-self-affine sets.
\newblock {\em Monatsh. Math.}, 161(3):271--293, 2010.

\bibitem{Kempton2015}
T.~Kempton.
\newblock The scenery flow for self-affine measures.
\newblock preprint, arXiv:1505.01663, 2015.

\bibitem{LauNgai1999}
K.-S. Lau and S.-M. Ngai.
\newblock Multifractal measures and a weak separation condition.
\newblock {\em Adv. Math.}, 141(1):45--96, 1999.

\bibitem{LauNgaiWang2009}
K.-S. Lau, S.-M. Ngai, and X.-Y. Wang.
\newblock Separation conditions for conformal iterated function systems.
\newblock {\em Monatsh. Math.}, 156(4):325--355, 2009.

\bibitem{LiWu2011}
J.~Li and M.~Wu.
\newblock Pointwise dimensions of general {M}oran measures with open set
  condition.
\newblock {\em Sci. China Math.}, 54(4):699--710, 2011.

\bibitem{Mattila1995}
P.~Mattila.
\newblock {\em Geometry of Sets and Measures in Euclidean Spaces: Fractals and
  Rectifiability}.
\newblock Cambridge University Press, Cambridge, 1995.

\bibitem{MauldinUrbanski1996}
R.~D. Mauldin and M.~Urba{\'n}ski.
\newblock Dimensions and measures in infinite iterated function systems.
\newblock {\em Proc. London Math. Soc.}, 73(3):105--154, 1996.

\bibitem{Moran1946}
P.~A.~P. Moran.
\newblock Additive functions of intervals and {H}ausdorff measure.
\newblock {\em Proc. Cambridge Philos. Soc.}, 42:15--23, 1946.

\bibitem{RajalaVilppolainen2013}
T.~Rajala and M.~Vilppolainen.
\newblock Weakly controlled {M}oran constructions and iterated functions
  systems in metric spaces.
\newblock {\em Illinois J. Math.}, 55(3):1015--1051 (2013), 2011.

\bibitem{Shmerkin2011}
P.~Shmerkin.
\newblock Ergodic geometric measure theory.
\newblock unpublished, 2011.

\bibitem{Zerner1996}
M.~P.~W. Zerner.
\newblock Weak separation properties for self-similar sets.
\newblock {\em Proc. Amer. Math. Soc.}, 124(11):3529--3539, 1996.

\end{thebibliography}

\end{document}